\documentclass[a4paper,11pt]{article}
\usepackage{amsmath}
\usepackage{amsfonts}
\usepackage[latin9]{inputenc}
\usepackage[english]{varioref}
\usepackage{dcolumn}
\usepackage[height=22cm , width = 16cm , top = 3cm , left = 3cm, a4paper]{geometry}
\usepackage[a4paper]{geometry}
\usepackage[final]{graphicx}
\usepackage{epsfig}
\usepackage{pstricks}
\usepackage{psfrag}
\usepackage{rotating}
\usepackage{booktabs}
\usepackage{delarray}
\usepackage{amssymb}
\usepackage{rotating}
\usepackage{subfigure}
\usepackage{layout}
\usepackage{amsthm}
\usepackage{xcolor}
\usepackage{setspace}

{
	{
		{
			{

				\newtheorem{thm}{Theorem}[section]
				
				\newtheorem{lem}[thm]{Lemma}
				\newtheorem{propos}[thm]{Proposition}
				\newtheorem{rem}[thm]{Remark}

				\parindent 0pt
				\newcommand{\enter}{\bigskip}

\begin{document}
\thispagestyle{empty}
\author{Sanjiv Kumar Bariwal ${}^1$	\footnote{{\it{${}$ Corresponding author. \, Email address:}} p20190043@pilani.bits-pilani.ac.in}	, Rajesh Kumar ${}^2$\footnote{{\it{${}$ Email address:}} rajesh.kumar@pilani.bits-pilani.ac.in}\\
\footnotesize ${}^{1,2}$Department of Mathematics, Birla Institute of Technology and Science Pilani,\\ \small{ Pilani-333031, Rajasthan, India}\\
}
	\date{}																		
\title{{ Finite volume convergence analysis and error estimation for non-linear collisional induced breakage equation}}	
									
\maketitle
										
\begin{quote}
{\small {\em\bf Abstract}}: This article focuses on the finite volume method (FVM) as an instrument tool to deal with the non-linear collisional-induced breakage equation (CBE) that arises in the particulate process. Notably, we consider the non-conservative approximation of the CBE. The analysis of weak convergence of the approximated solutions under a feasible stability condition on the time step is investigated for locally bounded breakage and collision kernels. Subsequently, explicit error estimation of the FVM solutions in uniform mesh having the kernels in the class of $W_{loc}^{1,\infty}$ space. It is also shown numerically for the first-order convergent scheme by taking numerical examples.
	\end{quote}
\textbf{Keywords}: Non-linear breakage, Finite volume method, Weak Solution, Convergence, Error.\\

{\bf{Mathematics Subject Classification (2020) 45L05, 45K05, 65R10}}\\

\section{Introduction}
Particulate processes are prominent in particle evolution dynamics and portray  how particles might unite to generate extensive size of particle or break into tiny particles. These particles  are labelled as   
  Particle is entirely defined by a single size variable, such as its volume or mass. There are two types of particle breakage or fragmentation: linear breakage and nonlinear collision-induced breakage. The effectiveness of the linear breakage equation is well recognized in the study of significant phenomena in a variety of scientific fields, including engineering; for further references, see \cite{diemer2021applications, danha2015application}. It is thought that in order to increase the range of operations that can be assessed, it must be expanded. One conceivable expansion is to include nonlinearity in the breaking process which can occur when the fragmentation behaviour of a particle is determined not only by its characteristics and dynamic circumstances (as in linear breakage), but also by the state and properties of the entire system, i.e., by binary interactions. Collisional breakage could enable some mass transfer between colliding particles. As a result, daughter particles with more extensive volumes than the parent particles are generated.  Non-linear models emerge in a wide range of contexts, including milling and crushing processes \cite{spampinato2017modelling,chen2020collision}, bulk distribution of asteroids \cite{kudzotsa2013mechanisms,yano2016explosive}, fluidized beds \cite{cruger2016coefficient,lee2017development}, etc. \\

The following partial integro-differential equation was used by Cheng and Redner \cite{cheng1988scaling} to obtain the binary collisional breakage equation, and is expressed the evolution of particles in terms of volume $x \in ]0,\infty[$ at time $t\geq 0$ and is defined by
\begin{align}\label{maineq}
\frac{\partial{\mathcal{C}(\mathrm{t},\mathrm{m})}}{\partial t}= \int_0^\infty\int_{\mathrm{m}}^{\infty} \mathbb{K}(\mathrm{n},z)\mathbb{B}(\mathrm{m},\mathrm{n},z)\mathcal{C}(t,\mathrm{\mathrm{n}})\mathcal{C}(t,z)\,d\mathrm{n}\,dz -\int_{0}^{\infty}\mathbb{K}(\mathrm{m},\mathrm{n})\mathcal{C}(t,\mathrm{m})\mathcal{C}(t,\mathrm{n})\,d\mathrm{n}
\end{align}
with the initial data
\begin{align}\label{initial}
\mathcal{C}(0,\mathrm{m})\ \ = \ \ \mathcal{C}^{in}(\mathrm{m}) \geq 0, \ \ \ \mathrm{m} \in \mathbb{R}^{+}:= ]0,\infty[.
\end{align}
In Eq.(\ref{maineq}), $\mathcal{C}(t,\mathrm{m})$  is the concentration function with volume between $\mathrm{m}$ and $\mathrm{m}+d\mathrm{m}$, and the collision kernel $\mathbb{K}(\mathrm{m},\mathrm{n})$ depicts the collision frequency  of particles of  $\mathrm{m}$ and $\mathrm{n}$ volume for breakage event. In practice, $\mathbb{K}(\mathrm{m},\mathrm{n}) \geq 0$ and $\mathbb{K}(\mathrm{m},\mathrm{n})=\mathbb{K}(\mathrm{n},\mathrm{m})$. The term  $\mathbb{B}(\mathrm{m},\mathrm{n},z)$ is called the breakage distribution function, which defines the rate for production of  $\mathrm{m}$ volume particles by splitting of  $\mathrm{n}$ volume particle due to interaction with $z$. The function $\mathbb{B}$ holds
\begin{align}\label{breakagefunc}
\int_{0}^{\mathrm{n}}\mathrm{m}\mathbb{B}(\mathrm{m},\mathrm{n},z)d\mathrm{m}= \mathrm{n}, \,\,\mathbb{B}(\mathrm{m},\mathrm{n},z)\neq 0\,\,\, \text{for} \hspace{0.4cm} \mathrm{m} \in (0,\mathrm{n})\,\,\, \text{and} \hspace{0.4cm} \mathbb{B}(\mathrm{m},\mathrm{n},z)=0 \,\,\,\text{for} \hspace{0.4cm} \mathrm{m}> \mathrm{n}.
\end{align}
The first term in Eq.(\ref{maineq}) explains gaining $\mathrm{m}$ volume particles due to  collisional breakage  between  $\mathrm{n}$ and $z$ volume particles, known as the birth term. The second term is labeled as the death term and describes the disappearance of  $\mathrm{m}$ volume particles due to collision with particles of volume $\mathrm{n}$. \\

It is also necessary to specify some integral features of the concentration function $\mathcal{C}(t,\mathrm{m})$, known as moments. The following equation defines the $j^{th}$ moment of the solution as
\begin{align}\label{momemt}
M_{j}(t)=\int_{0}^{\infty}\mathrm{m}^{j}\mathcal{C}(t,\mathrm{m})\,d\mathrm{m},
\end{align}
where the zeroth $M_{0}(t)$ and first $M_{1}(t)$ moments are equal to the total number and volume of particles in the system, respectively.
 Here, $M_{1}^{in} < \infty$ conveys the initial  particles volume  in the closed particulate system.\\

Before venturing into the specifics of the current work, let us reanalyze the existing literature on the linear breakage equation that has been widely investigated over the years concerning its analytical solutions \cite{ziff1991new,ziff1985kinetics}, similarity solutions \cite{peterson1986similarity,breschi2017note}, numerical results \cite{ hosseininia2006numerical,catak2010discrete,liao2018discrete} and references therein. There is a substantial body of work on the well-posedness of the coagulation and linear breakage equations (CLBE). The authors examined the existence of mass-conserving solutions to CLBE by having non-conservating approximations for unbounded coagulation and the breakage functions in \cite{barik2018note}. Moreover, many publications are accessible for solving coagulation-fragmentation equations numerically, including the method of moments \cite{attarakih2009solution}, finite element scheme \cite{ahmed2013stabilized}, Monte Carlo methodology \cite{lin2002solution}, and  finite volume method (FVM)\cite{bourgade2008convergence,bariwal2023convergence,bariwal2023numerical} to name a few. \\

Collisional breakage model is discussed in just a few mathematical articles in the literature, see \cite{cheng1988scaling,kostoglou2000study,laurenccot2001discrete,walker2002coalescence, barik2020global,das2022existence}. In \cite{cheng1988scaling}, the general behaviour of CBE is reviewed for homogeneous breakage functions. The model's analytical and self-similar solutions are discovered in \cite{kostoglou2000study} for some combinations of breakage and collision kernels.
The authors explained the global classical solutions of coagulation and collisional equation with collision kernel, growing indefinitely for large volumes in \cite{barik2020global}. In addition, fewer publications in the physics literature \cite{cheng1990kinetics,ernst2007nonlinear,krapivsky2003shattering} are committed to the CBE, with the majority dealing with scaling behavior and shattering transitions.  Lauren\c{c}ot and Wrzosek \cite{laurenccot2001discrete} investigated the existence and uniqueness of a solution for coagulation with CBE in which they have used the following constraint over the kernels 
$$\mathbb{K}(\mathrm{m},\mathrm{n}) \leq (\mathrm{m}\mathrm{n})^{\alpha}, \,\, \alpha\in [0,1) \,\, \text{and}\,\, \mathbb{B}(\mathrm{m},\mathrm{n},z)\leq P< \infty,\,\, \text{for}\,\, 1\leq \mathrm{m}< \mathrm{n}.$$
They have also explored gelation and the long-term behavior of solutions. Further, in \cite{giri2021weak}, the analysis of weak solution is developed with the coagulation dominating process where the collision kernel grows at most linearly at infinity. \\

Currently, in the numerical sense, it has been reported by several authors that the FVM is an appropriate option among the other  numerical techniques  for solving aggregation and  breakage  equations due to its mass conservation property. Therefore, FVM is presented for collision-induced breakage equation to analyze the concentration function.

 None of the previous studies have taken into account the non-linear collisional breakage equation, primarily due to the complexity of the model and the analytical investigation involved. In this regard, it is mandatory to implement a numerical method to observe the concentration function for physical relevant kernels. Here we perform the weak convergence of the non-conservative numerical scheme (FVM) for solving the collisional breakage equation. Therefore, this article is an attempt to study the weak convergence analysis of the FVM for solving the model with nonsingular unbounded kernels and then error estimation for kernels being in $W_{loc}^{1,\infty}$ space over a uniform mesh. Thanks to the idea taken from  Bourgade and Filbet \cite{bourgade2008convergence}, in which they have treated coagulation and binary fragmentation equation. The proof employs the weak $L^1$ compactness approach.\\

To proceed further, firstly, we will concentrate on the functional setting, having in mind that the initial concentration function contains the properties of mass conservation and finiteness of total particles in the system. Therefore, we construct the solution space that exhibits the convergence of the discretized numerical solution to the weak solution of the CBE (\ref{maineq}), and is recognized as a weighted $L^1$ space such as 
	\begin{align}\label{Space}
		X^{+} = \{\mathcal{C}\in L^{1}(\mathbb{R}^{+})\cap L^1(\mathbb{R}^{+}, \mathrm{m}\,d\mathrm{m}): \mathcal{C}\geq 0, \|\mathcal{C}\|< \infty\},
	\end{align}
where $\|\mathcal{C}\|= \int_{0}^{\infty}(1+\mathrm{m})\mathcal{C}(\mathrm{m})\,d\mathrm{m},$ and $\mathcal{C}^{in} \in X^{+}$, i.e., $M_{1}^{in}$ is a finite quantity. \\
	
Now, the specifications of the kernels  $\mathbb{B}$ and $\mathbb{K}$ are expressed in the following expression: both functions are  measurable over the domain. There exist  $\zeta, \eta$  with $0<\zeta \leq \eta \leq 1, \, \zeta+ \eta \leq 1 $ and $\alpha\geq 0$, $\lambda >0$ such that 	
\begin{align}\label{breakage funcn}
H1:  \hspace{0.2cm}	\mathbb{B}\in L_{\text{loc}}^\infty{(\mathbb{R}^{+} \times \mathbb{R}^{+} \times \mathbb{R}^{+}) },
\end{align}
$H2:$  
\begin{equation}\label{Collisional func}
\mathbb{K}(\mathrm{m},\mathrm{n})=\left\{
 \begin{array}{ll}
				           
 \lambda \mathrm{m}\mathrm{n} & \quad (\mathrm{m},\mathrm{n})\in ]0,1[\times ]0,1[\\
 \lambda \mathrm{m}\mathrm{n}^{-\alpha} & \quad (\mathrm{m},\mathrm{n})\in ]0,1[\times ]1,\infty[\\
 \lambda \mathrm{m}^{-\alpha}\mathrm{n} & \quad (\mathrm{m},\mathrm{n})\in ]1,\infty[\times ]0,1[\\
\lambda(\mathrm{m}^{\zeta}\mathrm{n}^{\eta}+\mathrm{m}^{\eta}\mathrm{n}^{\zeta}) & \quad (\mathrm{m},\mathrm{n})\in ]1,\infty[\times ]1,\infty[. 
 \end{array}
\right.
\end{equation}

For the sake of understanding the breakage function, those are locally bounded \cite{das2020approximate}, i.e., $\mathbb{B}(\mathrm{m},\mathrm{n},z)=\delta(\mathrm{m}-0.4\mathrm{n})+\delta(\mathrm{m}-0.6\mathrm{n})$, where both colliding particles will fragment into two particles of size $40\%$ and $60\%$ of the parent particles. One more example is 
\begin{equation*}
\mathbb{B}(\mathrm{m},\mathrm{n},z)=\left\{
 \begin{array}{ll}	           
 \frac{2}{\mathrm{n}} & \quad \text{if}\,\, \mathrm{n} >z,\\
 \delta(\mathrm{m}-\mathrm{n}) & \quad \text{if} \,\, \mathrm{n} \leq z, 
 \end{array}
\right.
\end{equation*}
where only large size particle breaks into two smaller size particles.\\

This article's contents are organized as follows. The discretization methodology based on the FVM and non-conservative form of fully discretized CBE are introduced in Section \ref{scheme}. Next, Section \ref{convergence} equips the convergence study of weak solution in $L^1$ space. In Section \ref{Error}, error analysis of FVM solution with first-order accuracy is described for uniform meshes. Additionally, we have justified  the theoretical error estimation via numerical results in Section \ref{testing}. Consequently, in the final section, some  conclusions are presented.
\section{Numerical Scheme}\label{scheme}
In this section, we commence exploring the FVM for the solution of Eq.(\ref{maineq}). It is based on the spatial domain being divided into tiny grid cells.
 Particle volumes ranging from $0$ to $\infty$ are taken into account in Eq.(\ref{maineq}). Nevertheless, we define the particle volumes to be in a finite domain for practical purposes. The scheme considers $]0, \mathfrak{R}]$ as truncated domain  with $0<R <\infty$. Thus the collisional breakage equation is  truncated as 
		\begin{align}\label{trunceq}
				\frac{\partial{\mathcal{C}(t,\mathrm{m})}}{\partial t}=  \int_0^\mathfrak{R}\int_{\mathrm{m}}^{\mathfrak{R}} \mathbb{K}(\mathrm{n},z)\mathbb{B}(\mathrm{m},\mathrm{n},z)\mathcal{C}(t,\mathrm{n})\mathcal{C}(t,z)\,d\mathrm{n}\,dz -\int_{0}^{\mathfrak{R}}\mathbb{K}(\mathrm{m},\mathrm{n})\mathcal{C}(t,\mathrm{m})\mathcal{C}(t,\mathrm{n})\,d\mathrm{n},
					\end{align}
				and 
					\begin{align}\label{initial1}
								\mathcal{C}(0,\mathrm{m})\ \ = \ \ \mathcal{C}^{in}(\mathrm{m}) \geq 0, \ \ \ \mathrm{m} \in ]0,\mathfrak{R}].
					\end{align}	
Consider a partitioning of the operating domain $]0,\mathfrak{R}]$ into small cells as $\wp_a:=]\mathrm{m}_{a-1/2}, \mathrm{m}_{a+1/2}], \,\,a=1,2, \ldots, \mathfrak{I}$, where\,\, $\mathrm{m}_{1/2}=0, \ \ \mathrm{m}_{\mathfrak{I}+1/2}= \mathfrak{R}, \hspace{0.2cm} \Delta \mathrm{m}_a=\mathrm{m}_{a+1/2}-\mathrm{m}_{a-1/2}$ 
and  consider	$	\hslash= \text{max} \, \Delta \mathrm{m}_a \  \forall \ a. $ 	The grid points are the midpoints of each subinterval and are designated as $\mathrm{m}_a\,\, \forall\, a.$ 
Now, the expression of the mean value of the concentration function $\mathcal{C}_a(t)$ in the cell $\wp_a$ is determined  by 	
		\begin{align}\label{meandens}
							\mathcal{C}_{a}(t)=\frac{1}{\Delta \mathrm{m}_a}\int_{\mathrm{m}_{a-1/2}}^{\mathrm{m}_{a+1/2}}\mathcal{C}(t,\mathrm{m})\,d\mathrm{m}.
						\end{align}
	The time domain is limited to the range $[0,\mathfrak{T}]$ with $0 < \mathfrak{T} < \infty$, and it is discretized into $N$ time intervals with time step $\Delta t$. The interval is defined as 		
$$ \tau_n=[t_n,t_{n+1}[, \,\,  n=0,1,\ldots,N-1.$$ 
We now begin developing the scheme on non-uniform meshes. It has the big advantage of covering a larger area with fewer points than a uniform grid.
The discretization differs slightly from that of Filbet and Laurencot \cite{bourgade2008convergence},  where they first converted the model (\ref{maineq}) to a conservative equation using Leibniz integral rule, then discretized using FVM.
Here, in this work, we have used a scheme based on finite volume implementation from the continuous equation (\ref{maineq}). \\

 To derive the discretized version of the CBE (\ref{trunceq}), we  proceed as follows: integrating the Eq.(\ref{trunceq}) with respect to $\mathrm{m}$ over  $a^{th}$ cell yields the following discrete form
\begin{align}\label{semi}
\frac{d\mathcal{C}_a}{dt}=B_{\mathcal{C}}(a)-D_{\mathcal{C}}(a),
\end{align}
where
\begin{align*}
B_{\mathcal{C}}(a)=\frac{1}{\Delta \mathrm{m}_a}\int_{\mathrm{m}_{a-1/2}}^{\mathrm{m}_{a+1/2}}\int_0^{\mathrm{m}_{\mathfrak{I}+1/2}}\int_{\mathrm{m}}^{\mathrm{m}_{\mathfrak{I}+1/2}} \mathbb{K}(\mathrm{n},z)\mathbb{B}(\mathrm{m},\mathrm{n},z)\mathcal{C}(t,\mathrm{n})\mathcal{C}(t,z)d\mathrm{n}\,dz\,d\mathrm{m}
\end{align*}
\begin{align*}
D_{\mathcal{C}}(a)= \frac{1}{\Delta \mathrm{m}_a}\int_{\mathrm{m}_{a-1/2}}^{\mathrm{m}_{a+1/2}}\int_0^{\mathrm{m}_{\mathfrak{I}+1/2}} \mathbb{K}(\mathrm{m},\mathrm{n})\mathcal{C}(t,\mathrm{m})\mathcal{C}(t,\mathrm{n})d\mathrm{n}\,d\mathrm{m}
\end{align*}
along with initial distribution,
\begin{align}
\mathcal{C}_{a}(0)=\mathcal{C}_{a}^{in}=\frac{1}{\Delta \mathrm{m}_a}\int_{\mathrm{m}_{a-1/2}}^{\mathrm{m}_{a+1/2}}\mathcal{C}_{0}(\mathrm{m})\,d\mathrm{m}.
\end{align}	
Implementing the midpoint rule to all of the above representations provide the semi-discrete equation after some simplifications as, 
\begin{align}\label{semidiscrete}
\frac{d\mathcal{C}_{a}}{dt}=&\frac{1}{\Delta \mathrm{m}_a}\sum_{l=1}^{\mathfrak{I}}\sum_{j=a}^{\mathfrak{I}}\mathbb{K}_{j,l}\mathcal{C}_{j}(t)\mathcal{C}_{l}(t)\Delta \mathrm{m}_{j}\Delta \mathrm{m}_{l}\int_{\mathrm{m}_{a-1/2}}^{p_{j}^{a}}\mathbb{B}(\mathrm{m},\mathrm{m}_{j},\mathrm{m}_{l})\,d\mathrm{m}-\sum_{j=1}^{\mathfrak{I}}\mathbb{K}_{a,j}\mathcal{C}_{a}(t)\mathcal{C}_{j}(t)\Delta \mathrm{m}_{j}, 
\end{align}
where the term $p_{j}^{a}$ is expressed by
\begin{equation}
p_{j}^{a} =
\begin{cases}
\mathrm{m}_{a}, & \text{if }\,j=a \\
\mathrm{m}_{a+1/2}, & j\neq a.							
\end{cases}
\end{equation}
Now, to obtain a fully discrete system, applying explicit Euler discretization to time variable $t$ leads to
\begin{align}\label{fully}
\mathcal{C}_{a}^{n+1}-\mathcal{C}_{a}^{n}=\frac{\Delta t}{\Delta \mathrm{m}_{a}}\sum_{l=1}^{\mathfrak{I}}\sum_{j=a}^{\mathfrak{I}}\mathbb{K}_{j,l}\mathcal{C}_{j}^{n}\mathcal{C}_{l}^{n}\Delta \mathrm{m}_{j}\Delta \mathrm{m}_{l}\int_{\mathrm{m}_{a-1/2}}^{p_{j}^{a}}\mathbb{B}(\mathrm{m},\mathrm{m}_{j},\mathrm{m}_{l})\,d\mathrm{m}  
-\Delta t \sum_{j=1}^{\mathfrak{I}}\mathbb{K}_{a,j}\mathcal{C}_{a}^{n}\mathcal{C}_{j}^{n}\Delta \mathrm{m}_{j}.
\end{align}
For the convergence analysis, consider a function $\mathcal{C}^{\hslash}$ on $[0,\mathfrak{T}]\times ]0,\mathfrak{R}]$ which is representated  by
\begin{align}\label{chap2:function_ch}
\mathcal{C}^{\hslash}(t,\mathrm{m})=\sum_{n=0}^{N-1}\sum_{a=1}^{\mathfrak{I}}\mathcal{C}_a^n\,\chi_{\wp_a}(\mathrm{m})\,\chi_{\tau_n}(t),
\end{align}
where $\chi_D(\mathrm{m})$ denotes the characteristic function on a set $D$ as $\chi_D(\mathrm{m})=1$ if $\mathrm{m}\in D$ or $0$ everywhere else. Also noting that $$\mathcal{C}^{\hslash}(0,\cdot)=\sum_{a=1}^{\mathfrak{I}}\mathcal{C}_a^{in} \chi_{\wp_a}(\cdot) \rightarrow \mathcal{C}^{in} \in L^1 ]0,\mathfrak{R}[\,\, \text{as} \,\, \hslash\rightarrow 0.$$ 
The convergence result relies on the correct approximation of the kernels through finite volume on the computational domain $u,v,w \in (0,\mathfrak{R}]$ and is given as 
\begin{align}\label{chap2:function_aggregatediscrete}
	\mathbb{K}^{\hslash}(u,v)= \sum_{a=1}^{\mathfrak{I}} \sum_{j=1}^{\mathfrak{I}} \mathbb{K}_{a,j} \chi_{\wp_a}(u) \chi_{\wp_j}(v),
\end{align}
\begin{align}\label{chap2:function_brkdiscrete}
	\mathbb{B}^{\hslash}(u,v,w)= \sum_{a=1}^{\mathfrak{I}} \sum_{j=1}^{\mathfrak{I}}\sum_{l=1}^{\mathfrak{I}} \mathbb{B}_{a,j,l} \chi_{\wp_a}(u) \chi_{\wp_j}(v)\chi_{\wp_l}(w),
	\end{align}

where $$\mathbb{K}_{a,j}= \frac{1}{\Delta \mathrm{m}_a \Delta \mathrm{m}_j} \int_{\wp_j} \int_{\wp_a} \mathbb{K}(u,v)du\,dv, \quad  \mathbb{B}_{a,j,l}= \frac{1}{\Delta \mathrm{m}_a \Delta \mathrm{m}_j \Delta \mathrm{m}_l}\int_{\wp_l} \int_{\wp_j} \int_{\wp_a} \mathbb{B}(u,v,w)du\,dv\,dw.$$
These kernel discretizations confirm the strong convergence of itself in $L^1$ space, i.e., $\mathbb{K}^{\hslash}\rightarrow \mathbb{K}$ and $\mathbb{B}^{\hslash} \rightarrow \mathbb{B}$ as $\hslash\rightarrow 0$; see \cite{bourgade2008convergence}.

\section{Weak Convergence}\label{convergence}
The objective of this section is to study the convergence of solution $\mathcal{C}^{\hslash}$ to a function $\mathcal{C}$  as $\hslash$ and $\Delta t$ $\rightarrow 0$. 

\begin{thm}\label{maintheorem}
Assume that the kernels hold the properties in (\ref{breakage funcn}-\ref{Collisional func}) and $\mathcal{C}^{in}$ satisfy (\ref{Space}). Furthermore, suppose that the time step $\Delta t$ confirms the existence of a positive constant $\theta$ such that
						\begin{align}\label{22}
							S(\mathfrak{T}, \mathfrak{R})\Delta t\le \theta< 1,
						\end{align}
						holds for 
						\begin{align}{\label{23}}
							S(\mathfrak{T}, \mathfrak{R}):= \lambda(2\mathfrak{R}  \|\mathcal{C}^{in}\|_{L^1}\,e^{2\lambda \mathfrak{R} \|\mathbb{B}\|_{L^{\infty}} M_{1}^{in} \mathfrak{T}} + M_{1}^{in}).
						\end{align}
						Then assurance of a subsequence arises that
					 $$\mathcal{C}^{\hslash}\rightarrow \mathcal{C}\ \
						\text{in}\ \ L^\infty([0,\mathfrak{T}];L^1\,]0,\mathfrak{R}[),$$
						  where  $\mathcal{C}$ is the weak solution of problem  (\ref{maineq}-\ref{initial}) on $[0,\mathfrak{T}]$. This implies that, the function $\mathcal{C}\geq 0$ satisfies
						\begin{align}\label{convergence0}
									\int_0^\mathfrak{T} &\int_0^{\mathfrak{R}} \mathcal{C}(t,\mathrm{m})\frac{\partial\varphi}{\partial
									t}(t,\mathrm{m})d\mathrm{m}\,dt -\int_0^\mathfrak{T}\int_0^{\mathfrak{R}} \int_0^{\mathfrak{R}} \int_{\mathrm{m}}^{\mathfrak{R}}\varphi(t,\mathrm{m})\mathbb{K}(\mathrm{n},z)\mathbb{B}(\mathrm{m},\mathrm{n},z)\mathcal{C}(t,\mathrm{n})\mathcal{C}(t,z)d\mathrm{n}\,dz\,d\mathrm{m}\,dt\nonumber\\
								&+\int_0^{\mathfrak{R}} \mathcal{C}^{in}(\mathrm{m})\varphi(0,\mathrm{m})d\mathrm{m}+ \int_0^\mathfrak{T} \int_0^{\mathfrak{R}} \int_{0}^{\mathfrak{R}}\varphi(t,\mathrm{m})\mathbb{K}(\mathrm{m},\mathrm{n})\mathcal{C}(t,\mathrm{m})\mathcal{C}(t,\mathrm{n})d\mathrm{n}\,d\mathrm{m}\,dt
							=0,
						\end{align}
						where $\varphi$ is compactly supported smooth functions on $[0,\mathfrak{T}]\times ]0,{\mathfrak{R}}].$
					\end{thm}
					The main goal, based on the preceding theorem, is to show the weak convergence of the family of functions 
					$(\mathcal{C}^{\hslash})$ to 
					$\mathcal{C}$ in 
				$L^1]0,\mathfrak{R}[$ as $\hslash$ and $\Delta t$ approach zero, using the Dunford-Pettis theorem to establish compactness under weak convergence.

	\begin{thm}\label{maintheorem1}
	Assuming a sequence $\mathcal{C}^{\hslash}:\Omega\mapsto \mathbb{R} $ in $L^1(\Omega)$ with $|\Omega|<\infty$ and satisfies the following properties
	\begin{enumerate}
	\item Equiboundedness of $\{\mathcal{C}^{\hslash}\}$, i.e.,
		\begin{align}\label{equiboundedness}
	\sup \|\mathcal{C}^{\hslash}\|_{L^1(\Omega)}< \infty
	\end{align}
	\item Equiintegrability of $\{\mathcal{C}^{\hslash}\}$, iff
	\begin{align}\label{equiintegrable}
	\int_\Omega \psi(|\mathcal{C}^{\hslash}|)d\mathrm{m}< \infty,
	\end{align}
	where  $\psi:[0,\infty[\mapsto [0,\infty[$ is a increasing function with
	\begin{align*}
	\lim_{\nu\rightarrow \infty}\frac{\psi(\nu)}{\nu}\rightarrow
								\infty.
		\end{align*}
			\end{enumerate}
	Then $\mathcal{C}^{\hslash}$  weakly sequencly compact set in $L^1(\Omega)$, see \cite{laurenccot2002continuous}.
	\end{thm}
In order to prove Theorem \ref{maintheorem1}, firstly, discuss the non-negativity and equiboundedness of $\mathcal{C}^{\hslash}$ in $L^1$ and define the midpoint approximation of $m$ as follows:
		\begin{align*}
			X^\hslash:\mathrm{m}\in (0,R)\rightarrow
			X^\hslash(\mathrm{m})=\sum_{a=1}^{\mathfrak{I}}\mathrm{m}_a\chi_{\wp_a}(\mathrm{m}).
			\end{align*}
	
\begin{propos}
{Let us assume that time step satisfy condition (\ref{22}) and growth condition of kernels hold (\ref{breakage funcn}-\ref{Collisional func}). Then the non-negative family of $\mathcal{C}^{\hslash}$ fulfills
\begin{align}
\int_0^\mathfrak{R} X^\hslash(\mathrm{m})\mathcal{C}^{\hslash}(t,\mathrm{m})d\mathrm{m} \leq \int_0^\mathfrak{R} X^\hslash(\mathrm{m})\mathcal{C}^{\hslash}(s,\mathrm{m})d\mathrm{m} \leq \int_0^\mathfrak{R} X^\hslash(\mathrm{m})\mathcal{C}^{\hslash}(0,\mathrm{m})d\mathrm{m}=:M_1^{in},
\end{align}
where $ 0\le s\le t\le \mathfrak{T},$ and 
\begin{align}\label{36}
\int_0^\mathfrak{R} \mathcal{C}^{\hslash}(t,\mathrm{m})d\mathrm{m} \leq  \|\mathcal{C}^{in}\|_{L^1}\,e^{2\lambda \mathfrak{R} \|\mathbb{B}\|_{L^{\infty}} M_{1}^{in} t}.
\end{align}
}
\end{propos}
\begin{proof}

Mathematical induction is used to demonstrate the non-negativity of family $\mathcal{C}^{\hslash}$. At $t = 0$, it is known that $\mathcal{C}^{\hslash}(0) \geq 0$ belongs to $L^1]0,\mathfrak{R}[$. Let us assume that $\mathcal{C}^{\hslash}(t^n)\geq 0 $ and
 	\begin{align}\label{mon1}
 	\int_0^{\mathfrak{R}} \mathcal{C}^{\hslash}(t^n,\mathrm{m})d\mathrm{m} \le  \|\mathcal{C}^{in}\|_{L^1}\,e^{2\lambda \mathfrak{R} \|\mathbb{B}\|_{L^{\infty}} M_{1}^{in} t^n}.
 	\end{align}
 	Now to justify $\mathcal{C}^{\hslash}(t^{n+1})\geq 0$, start the indexing from $a = 1$. As a result, in this situation, we obtain from Eq.(\ref{fully}),
 	\begin{align}\label{non}
 	\mathcal{C}_{1}^{n+1}=& \mathcal{C}_{1}^{n}+\frac{\Delta t}{\Delta \mathrm{m}_{1}}\sum_{l=1}^{\mathfrak{I}}\sum_{j=1}^{\mathfrak{I}}\mathbb{K}_{j,l}\mathcal{C}_{j}^{n}\mathcal{C}_{l}^{n}\Delta \mathrm{m}_{j}\Delta \mathrm{m}_{l}\int_{\mathrm{m}_{1/2}}^{p_{j}^{1}}\mathbb{B}(\mathrm{m},\mathrm{m}_{j},\mathrm{m}_{l})\,d\mathrm{m}  
 	-\Delta t \sum_{j=1}^{\mathfrak{I}}\mathbb{K}_{1,j}\mathcal{C}_{1}^{n}\mathcal{C}_{j}^{n}\Delta \mathrm{m}_{j} \nonumber \\
 \geq 	& \mathcal{C}_{1}^{n}-\Delta t \sum_{j=1}^{\mathfrak{I}}\mathbb{K}_{1,j}\mathcal{C}_{1}^{n}\mathcal{C}_{j}^{n}\Delta \mathrm{m}_{j}.
 	\end{align}
	Moving further, we choose the first case for collisional kernel, case-(1):
	 $ \mathbb{K}(\mathrm{m},\mathrm{n})= \lambda(\mathrm{m}^{\zeta}\mathrm{n}^{\eta}+\mathrm{m}^{\eta}\mathrm{n}^{\zeta}),\, \text{when}  \,\, (\mathrm{m},\mathrm{n})\in ]1,\mathfrak{R}[\times ]1,\mathfrak{R}[.$ Thus, 
	 \begin{align*}
	 \mathcal{C}_{1}^{n+1}\geq  \mathcal{C}_{1}^{n}-\Delta t \sum_{j=1}^{\mathfrak{I}}\lambda(\mathrm{m}_{1}^{\zeta}\mathrm{m}_{j}^{\eta}+\mathrm{m}_{1}^{\eta}\mathrm{m}_{j}^{\zeta})\mathcal{C}_{1}^{n}\mathcal{C}_{j}^{n}\Delta \mathrm{m}_{j},
	 \end{align*}
	 using the fact that $\lambda(\mathrm{m}_{a}^{\zeta}\mathrm{m}_{j}^{\eta}+\mathrm{m}_{a}^{\eta}\mathrm{m}_{j}^{\zeta})\leq \lambda(\mathrm{m}_{a}+\mathrm{m}_{j})$, thanks to Young's inequality, we convert the above inequality into the following one
	 \begin{align}\label{non1}
	 \mathcal{C}_{1}^{n+1}&\geq  \,\mathcal{C}_{1}^{n}-\lambda\Delta t \sum_{j=1}^{\mathfrak{I}}(\mathrm{m}_{1}+\mathrm{m}_{j})\mathcal{C}_{1}^{n}\mathcal{C}_{j}^{n}\Delta \mathrm{m}_{j} \nonumber \\
	 &\geq [1-\lambda\Delta t(\mathfrak{R}\sum_{j=1}^{\mathfrak{I}}\mathcal{C}_{j}^{n}\Delta \mathrm{m}_{j}+ M_{1}^{in})]\mathcal{C}_{1}^{n}.
	  \end{align} 	
	  Now, consider case-(2): $\mathbb{K}(\mathrm{m},\mathrm{n})=\lambda \mathrm{m}^{-\alpha}\mathrm{n}, \, \text{when}\,\, (\mathrm{m},\mathrm{n})\in ]1,\mathfrak{R}[\times ]0,1[$. Putting this value in Eq.(\ref{non}) and then imposing the condition $\mathrm{m}^{-\alpha}\leq 1$ yield
	  \begin{align}
	   \mathcal{C}_{1}^{n+1}&\geq  \,\mathcal{C}_{1}^{n}-{\lambda}\Delta t \sum_{j=1}^{\mathfrak{I}}\mathrm{m}_{j}\mathcal{C}_{1}^{n}\mathcal{C}_{j}^{n}\Delta \mathrm{m}_{j}\nonumber  \\ 
	   & \geq (1-\lambda\Delta t M_{1}^{in})\mathcal{C}_{1}^{n}.
	  \end{align}
For case-(3): $\mathbb{K}(\mathrm{m},\mathrm{n})=\lambda \mathrm{m}\mathrm{n}^{-\alpha}, \, \text{when}\,\, (\mathrm{m},\mathrm{n})\in ]0,1[\times ]1,\mathfrak{R}[$ with $\mathrm{n}^{-\alpha}\leq 1$ and for case-(4): $\mathbb{K}(\mathrm{m},\mathrm{n})=\lambda (\mathrm{m}\mathrm{n}), \, \text{when}\,\, (\mathrm{m},\mathrm{n})\in ]0,1[\times ]0,1[$ provide
\begin{align}
\mathcal{C}_{1}^{n+1}\geq (1-\lambda\Delta t \sum_{j=1}^{\mathfrak{I}}\mathcal{C}_{j}^{n}\Delta \mathrm{m}_{j})\mathcal{C}_{1}^{n}.
\end{align}

All the results from case(1)-case(4) are collected, and the following inequality is achieved
\begin{align}
\mathcal{C}_{1}^{n+1}\geq [1-\lambda \Delta t(\mathfrak{R}\sum_{j=1}^{\mathfrak{I}}\mathcal{C}_{j}^{n}\Delta \mathrm{m}_{j}+M_{1}^{in})]\mathcal{C}_{1}^{n}.  
\end{align}
Using conditions (\ref{22}), (\ref{23}) and Eq.(\ref{mon1}), the non-negativity of $\mathcal{C}_{1}^{n+1}$ is obtained. Thus, we assume that the computations for $a\geq 2$ go similar to $a=1$ for all four cases and  the results are obtained like the previous ones. Then, applying the stability condition (\ref{22}) and the $L^1$ bound on $\mathcal{C}^{\hslash}$  yield $\mathcal{C}^{\hslash}({t^{n+1}})\geq 0.$\\
Further, by summing (\ref{fully}) over $a$, the following time monotonicity result is obtained for the total mass as  	
  	\begin{align}\label{Massloss}
  	\sum_{a=0}^{\mathfrak{I}}\Delta {\mathrm{m}_a} \mathrm{m}_a \mathcal{C}_a^{n+1} =& \sum_{a=0}^{\mathfrak{I}}\Delta \mathrm{m}_a \mathrm{m}_a \mathcal{C}_a^{n}+\Delta t\sum_{a=1}^{\mathfrak{I}}\sum_{l=1}^{\mathfrak{I}}\sum_{j=a}^{\mathfrak{I}}\mathrm{m}_a\mathbb{K}_{j,l}\mathcal{C}_{j}^{n}\mathcal{C}_{l}^{n}\Delta \mathrm{m}_{j}\Delta \mathrm{m}_{l}\int_{\mathrm{m}_{a-1/2}}^{p_{j}^{a}}\mathbb{B}(\mathrm{m},\mathrm{m}_{j},\mathrm{m}_{l})\,d\mathrm{m}  \nonumber \\	
  	&-\Delta t \sum_{a=1}^{\mathfrak{I}}\sum_{j=1}^{\mathfrak{I}}
  	\mathrm{m}_a \mathbb{K}_{a,j}\mathcal{C}_{a}^{n}\mathcal{C}_{j}^{n}\Delta \mathrm{m}_{j}\Delta \mathrm{m}_{a}.
  	\end{align}
  	Rearranging the summation's order and using the upper limit of $p_{j}^{a}$ and Eq.(\ref{breakagefunc}) help the second term of the  Eq.(\ref{Massloss}) on the right-hand side (RHS) yields
  	\begin{align}\label{Massloss1}
  	\Delta t\sum_{a=1}^{\mathfrak{I}}\sum_{l=1}^{\mathfrak{I}}\sum_{j=a}^{\mathfrak{I}}\mathrm{m}_a\mathbb{K}_{j,l}\mathcal{C}_{j}^{n}\mathcal{C}_{l}^{n}\Delta \mathrm{m}_{j}\Delta \mathrm{m}_{l}\int_{\mathrm{m}_{a-1/2}}^{p_{j}^{a}}\mathbb{B}(\mathrm{m},\mathrm{m}_{j},\mathrm{m}_{l})\,d\mathrm{m} \nonumber \\
  	\leq 	\Delta t\sum_{l=1}^{\mathfrak{I}}\sum_{j=1}^{\mathfrak{I}}\mathbb{K}_{j,l}\mathcal{C}_{j}^{n}\mathcal{C}_{l}^{n}\Delta \mathrm{m}_{j}\Delta \mathrm{m}_{l}\sum_{a=1}^{j}\mathrm{m}_a\int_{\mathrm{m}_{a-1/2}}^{\mathrm{m}_{a+1/2}}\mathbb{B}(\mathrm{m},\mathrm{m}_{j},\mathrm{m}_{l})\,d\mathrm{m} \nonumber \\
  	=	\Delta t\sum_{l=1}^{\mathfrak{I}}\sum_{j=1}^{\mathfrak{I}}\mathbb{K}_{j,l}\mathcal{C}_{j}^{n}\mathcal{C}_{l}^{n}\Delta \mathrm{m}_{j}\Delta \mathrm{m}_{l}\sum_{a=1}^{j}\mathrm{m}_a\mathbb{B}(\mathrm{m}_a,\mathrm{m}_{j},\mathrm{m}_{l})\Delta \mathrm{m}_{a} 	= \Delta t\sum_{l=1}^{\mathfrak{I}}\sum_{j=1}^{\mathfrak{I}}\mathrm{m}_j\mathbb{K}_{j,l}\mathcal{C}_{j}^{n}\mathcal{C}_{l}^{n}\Delta \mathrm{m}_{j}\Delta \mathrm{m}_{l}.
  	\end{align}
  Eqs.(\ref{Massloss}) and (\ref{Massloss1}) indicate about the approximate mass loss property
  \begin{align}
  	\sum_{a=0}^{\mathfrak{I}}\Delta {\mathrm{m}_a} \mathrm{m}_a \mathcal{C}_a^{n+1}  \leq  \sum_{a=0}^{\mathfrak{I}}\Delta \mathrm{m}_a \mathrm{m}_a \mathcal{C}_a^{n}\leq M_{1}^{in}.
  \end{align}	
  It is then demonstrated that 
 $\mathcal{C}^{\hslash}(t^{n+1})$ satisfies a similar estimate to (\ref{mon1}). This is done by multiplying  (\ref{fully}) by  $\Delta \mathrm{m}_i$, excluding the negative term, and summing over $a$, as
\begin{align}\label{equi1}
\sum_{a=1}^{\mathfrak{I}}\mathcal{C}_{a}^{n+1}\Delta \mathrm{m}_{a} &\leq  \sum_{a=1}^{\mathfrak{I}}\mathcal{C}_{a}^{n}\Delta \mathrm{m}_{a}+\Delta t\sum_{a=1}^{\mathfrak{I}}\sum_{l=1}^{\mathfrak{I}}\sum_{j=a}^{\mathfrak{I}}\mathbb{K}_{j,l}\mathcal{C}_{j}^{n}\mathcal{C}_{l}^{n}\Delta \mathrm{m}_{j}\Delta \mathrm{m}_{l}\int_{\mathrm{m}_{a-1/2}}^{p_{j}^{a}}\mathbb{B}(\mathrm{m},\mathrm{m}_{j},\mathrm{m}_{l})\,d\mathrm{m} \nonumber \\
& \leq \sum_{a=1}^{\mathfrak{I}}\mathcal{C}_{a}^{n}\Delta \mathrm{m}_{a}+
\Delta t\|\mathbb{B}\|_{\infty}\sum_{a=1}^{\mathfrak{I}}\sum_{l=1}^{\mathfrak{I}}\sum_{j=1}^{\mathfrak{I}}\mathbb{K}_{j,l}\mathcal{C}_{j}^{n}\mathcal{C}_{l}^{n}\Delta \mathrm{m}_{j}\Delta \mathrm{m}_{l}\int_{\mathrm{m}_{a-1/2}}^{\mathrm{m}_{a+1/2}}\,d\mathrm{m} \nonumber \\
& \leq \sum_{a=1}^{\mathfrak{I}}\mathcal{C}_{a}^{n}\Delta \mathrm{m}_{a}+
\Delta t \mathfrak{R} \|\mathbb{B}\|_{\infty}\sum_{l=1}^{\mathfrak{I}}\sum_{j=1}^{\mathfrak{I}}\mathbb{K}_{j,l}\mathcal{C}_{j}^{n}\mathcal{C}_{l}^{n}\Delta \mathrm{m}_{j}\Delta \mathrm{m}_{l}.
\end{align}
Again, the above will be simplified for four cases of kernels:\\
Case-(1): $ \mathbb{K}(\mathrm{m},\mathrm{n})= \lambda(\mathrm{m}^{\zeta}\mathrm{n}^{\eta}+\mathrm{m}^{\eta}\mathrm{n}^{\zeta}),\, \text{when}  \,\, (\mathrm{m},\mathrm{n})\in ]1,\mathfrak{R}[\times ]1,\mathfrak{R}[$. Substitute the value of $\mathbb{K}(\mathrm{m},\mathrm{n})$ in Eq.(\ref{equi1}) and use the Young's inequality to get  
\begin{align*}
\sum_{a=1}^{\mathfrak{I}}\mathcal{C}_{a}^{n+1}\Delta \mathrm{m}_{a} &\leq  \sum_{a=1}^{\mathfrak{I}}\mathcal{C}_{a}^{n}\Delta \mathrm{m}_{a} + {\lambda}\Delta t \mathfrak{R} \|\mathbb{B}\|_{\infty}\sum_{l=1}^{\mathfrak{I}}\sum_{j=1}^{\mathfrak{I}} (\mathrm{m}_{j}+\mathrm{m}_{l})\mathcal{C}_{j}^{n}\mathcal{C}_{l}^{n}\Delta \mathrm{m}_{j}\Delta \mathrm{m}_{l}\nonumber \\
& \leq (1+2{\lambda}\Delta t \mathfrak{R} \|\mathbb{B}\|_{\infty}M_{1}^{in})\sum_{a=1}^{\mathfrak{I}}\mathcal{C}_{a}^{n}\Delta \mathrm{m}_{a}.
\end{align*}
Finally, having (\ref{mon1}) the  $L^1$ bound of $\mathcal{C}^{\hslash}$ at time step $n$ and  $1+\mathrm{m} < \exp(\mathrm{m})$ $\forall$ $\mathrm{m}>0$ imply that
$$\sum_{a=1}^{\mathfrak{I}}\mathcal{C}_{a}^{n+1}\Delta \mathrm{m}_{a} \leq \|\mathcal{C}^{in}\|_{L^1}\,e^{2\lambda \mathfrak{R} \|\mathbb{B}\|_{L^{\infty}} M_{1}^{in} t^{n+1}}.$$
As a consequence, the result (\ref{36}) is accomplished.\\
Case-(2): $\mathbb{K}(\mathrm{m},\mathrm{n})=\lambda \mathrm{m}^{-\alpha}\mathrm{n}, \, \text{when}\,\, (\mathrm{m},\mathrm{n})\in ]1,\mathfrak{R}[\times ]0,1[$, and Case-(3): $\mathbb{K}(\mathrm{m},\mathrm{n})=\lambda \mathrm{m}\mathrm{n}^{-\alpha}, \, \text{when}\,\, (\mathrm{m},\mathrm{n})\in ]0,1[\times ]1,\mathfrak{R}[$ have  similar computations. The value of $\mathbb{K}(\mathrm{m},\mathrm{n})$ after  substituting in Eq.(\ref{equi1}) yields
\begin{align*}
\sum_{a=1}^{\mathfrak{I}}\mathcal{C}_{a}^{n+1}\Delta \mathrm{m}_{a} &\leq  \sum_{a=1}^{\mathfrak{I}}\mathcal{C}_{a}^{n}\Delta \mathrm{m}_{a} + {\lambda}\Delta t \mathfrak{R} \|\mathbb{B}\|_{\infty}\sum_{l=1}^{\mathfrak{I}}\sum_{j=1}^{\mathfrak{I}} (\mathrm{m}_{j}^{-\alpha}\mathrm{m}_{l})\mathcal{C}_{j}^{n}\mathcal{C}_{l}^{n}\Delta \mathrm{m}_{j}\Delta \mathrm{m}_{l}\\
&\leq \sum_{a=1}^{\mathfrak{I}}\mathcal{C}_{a}^{n}\Delta \mathrm{m}_{a} + {\lambda}\Delta t \mathfrak{R} \|\mathbb{B}\|_{\infty}\sum_{l=1}^{\mathfrak{I}}\sum_{j=1}^{\mathfrak{I}} \mathrm{m}_{l}\mathcal{C}_{j}^{n}\mathcal{C}_{l}^{n}\Delta \mathrm{m}_{j}\Delta \mathrm{m}_{l}\\
& \leq (1+\lambda\Delta t \mathfrak{R} \|\mathbb{B}\|_{\infty}M_{1}^{in})\sum_{a=1}^{\mathfrak{I}}\mathcal{C}_{a}^{n}\Delta \mathrm{m}_{a}.
\end{align*}
Again, using (\ref{mon1}) and $1+\mathrm{m} < \exp(\mathrm{m})$ $\forall$ $\mathrm{m}>0$  provide the $L^1$ bound for $\mathcal{C}^{\hslash}$ at time step $n+1$.

Case-(4): For  $\mathbb{K}(\mathrm{m},\mathrm{n})=\lambda (\mathrm{m}\mathrm{n}), \, \text{when}\,\, (\mathrm{m},\mathrm{n})\in ]0,1[\times ]0,1[$,  inserting the value of $\mathbb{K}$ in Eq.(\ref{equi1}) employs
\begin{align*}
\sum_{a=1}^{\mathfrak{I}}\mathcal{C}_{a}^{n+1}\Delta \mathrm{m}_{a} &\leq  \sum_{a=1}^{\mathfrak{I}}\mathcal{C}_{a}^{n}\Delta \mathrm{m}_{a} + {\lambda}\Delta t \mathfrak{R} \|\mathbb{B}\|_{\infty}\sum_{l=1}^{\mathfrak{I}}\sum_{j=1}^{\mathfrak{I}} \mathrm{m}_{j}\mathrm{m}_{l}\mathcal{C}_{j}^{n}\mathcal{C}_{l}^{n}\Delta \mathrm{m}_{j}\Delta \mathrm{m}_{l}\\
& \leq \sum_{a=1}^{\mathfrak{I}}\mathcal{C}_{a}^{n}\Delta \mathrm{m}_{a} + {\lambda}\Delta t \mathfrak{R} \|\mathbb{B}\|_{\infty}\sum_{l=1}^{\mathfrak{I}}\sum_{j=1}^{\mathfrak{I}} \mathrm{m}_{l}\mathcal{C}_{j}^{n}\mathcal{C}_{l}^{n}\Delta \mathrm{m}_{j}\Delta \mathrm{m}_{l}.
\end{align*}
 To get the result (\ref{36}) for $\mathcal{C}^{\hslash}(t^{n+1}),$ the computations are similar to the previous case.
\end{proof}
Now, the uniform integrability can be demonstrated by considering a class of convex functions as $C_{V P,\infty}$. A non-negative function $ \psi \in C_{V P,\infty}$,
 holds following conditions \cite{laurenccot2002continuous}:

	\begin{description}
	\item[(i)] $\psi(0)=0,\ \psi'(0)=1$ and $\psi'$ is concave;
	\item[(ii)] $\lim_{\nu \to \infty} \psi'(\nu) =\lim_{\nu \to \infty} \frac{ \psi(\nu)}{\nu}=\infty$;
	\item[(iii)] for $\theta \in ]1, 2[$,
	\begin{align}\label{Tproperty}
	\Pi_{\theta}(\psi):= \sup_{\nu \ge 0} \bigg\{   \frac{ \psi(\nu)}{\nu^{\theta}} \bigg\} < \infty.
	\end{align}
	\end{description}
	
It is given that $\mathcal{C}^{in}\in L^1\,]0,\mathfrak{R}[$. Thus, by the De la Vallée Poussin theorem, there exists a convex function $\mathrm{\psi}\geq 0$, continuously differentiable on $\mathbb{R}^{+}$, with $\mathrm{\psi}(0)=0$,$\mathrm{\psi}'(0)=1$, and $\mathrm{\psi}'$ concave
	$$\frac{\mathrm{\psi}(\nu)}{\nu} \rightarrow \infty,\ \ \text{as}\ \
	\nu \rightarrow \infty$$ and
	\begin{align}\label{convex}
	\mathcal{I}:=\int_0^{\mathfrak{R}} \mathrm{\psi}(\mathcal{C}^{in})(\mathrm{m})d\mathrm{m}< \infty.
	\end{align}
	\begin{lem} [\cite{laurenccot2002continuous}]\label{lemma}
	Consider $\mathrm{\psi}\in {C}_{VP, \infty}$. 
Then $\forall$  $(z_1,z_2)\in \mathbb{R}^{+}\times \mathbb{R}^{+},$
	$$z_1\mathrm{\psi}'(z_2)\leq \mathrm{\psi}(z_1)+\mathrm{\psi}(z_2).$$
	\end{lem}
	At this point, we can demonstrate that the approximate solution $\mathcal{C}^\hslash$ is equiintegrable.
	\begin{propos}\label{equiintegrability}
	Assume that $\mathcal{C}^{in}\geq 0\in L^1 ]0,\mathfrak{R}[$ and (\ref{fully}) produces the family $(\mathcal{C}^{\hslash})$ for any $h$ and $\Delta t$, where $\Delta t$ holds (\ref{22}). Then $(\mathcal{C}^{\hslash})$ is weakly sequentially compact in $L^1(]0,\mathfrak{T}[\times ]0,\mathfrak{R}[)$.
						\end{propos}
						
	\begin{proof}
	We shall establish an estimate for the family $\mathcal{C}^{\hslash}$, uniformly on $\hslash$, similar to (\ref{convex}), in order to demonstrate the equiintegrability. In this regard, define the integral of $\mathrm{\psi}(\mathcal{C}^{\hslash})$ as
		\begin{align*}
		\int_0^\mathfrak{T} \int_0^{\mathfrak{R}} \mathrm{\psi}(\mathcal{C}^{\hslash}(t,\mathrm{m}))d\mathrm{m}\,dt=&\sum_{n=0}^{N-1}\sum_{a=1}^{\mathfrak{I}}\int_{\tau_n}\int_{\wp_a}\mathrm{\psi}
		\bigg(\sum_{k=0}^{N-1}\sum_{j=1}^{\mathfrak{I}}\mathcal{C}_j^k\chi_{\wp_j}(\mathrm{m})\chi_{\tau_k}(t)\bigg)d\mathrm{m}\,dt\\
		=&\sum_{n=0}^{N-1}\sum_{a=1}^{\mathfrak{I}}\Delta t\Delta \mathrm{m}_a\mathrm{\psi}(\mathcal{C}_a^n).
		\end{align*}
	It follows from the discrete Eq.(\ref{fully}), as well as the convex property of  $\mathrm{\psi}$  and  ${\mathrm{\psi}}^{'}\geq 0$, that
	\begin{align}\label{Equi1}
\sum_{a=1}^{\mathfrak{I}} [\mathrm{\psi}(\mathcal{C}_a^{n+1})-\mathrm{\psi}(\mathcal{C}_a^{n})]\Delta \mathrm{m}_{a}& \leq \sum_{a=1}^{\mathfrak{I}}\left(\mathcal{C}_a^{n+1}-\mathcal{C}_a^{n}\right)\mathrm{\psi}^{'}(\mathcal{C}_a^{n+1})\Delta \mathrm{m}_{a} \nonumber\\
& \leq \Delta t\sum_{a=1}^{\mathfrak{I}}\sum_{l=1}^{\mathfrak{I}}\sum_{j=a}^{\mathfrak{I}}\mathbb{K}_{j,l}\mathcal{C}_{j}^{n}\mathcal{C}_{l}^{n}\mathrm{\psi}^{'}(\mathcal{C}_a^{n+1})\Delta \mathrm{m}_{j}\Delta \mathrm{m}_{l}\int_{\mathrm{m}_{a-1/2}}^{\mathrm{m}_{a+1/2}}\mathbb{B}(\mathrm{m},\mathrm{m}_{j},\mathrm{m}_{l})\,d\mathrm{m} \nonumber \\
&  \leq \Delta t\sum_{a=1}^{\mathfrak{I}}\sum_{l=1}^{\mathfrak{I}}\sum_{j=1}^{\mathfrak{I}}\mathbb{K}_{j,l}\mathcal{C}_{j}^{n}\mathcal{C}_{l}^{n}\mathrm{\psi}^{'}(\mathcal{C}_a^{n+1})\Delta \mathrm{m}_{j}\Delta \mathrm{m}_{l} \mathbb{B}(\mathrm{m}_{a},\mathrm{m}_{j},\mathrm{m}_{l})\Delta \mathrm{m}_{a}.
	\end{align}
Case-(1): $ \mathbb{K}(\mathrm{m},\mathrm{n})= \lambda(\mathrm{m}^{\zeta}\mathrm{n}^{\eta}+\mathrm{m}^{\eta}\mathrm{n}^{\zeta}),\, \text{when}  \,\, (\mathrm{m},\mathrm{n})\in ]1,\mathfrak{R}[\times ]1,\mathfrak{R}[$. Substitute the value of $\mathbb{K}(\mathrm{m},\mathrm{n})$ in Eq.(\ref{Equi1}) yields
\begin{align*}
\sum_{a=1}^{\mathfrak{I}} [\mathrm{\psi}(\mathcal{C}_a^{n+1})-\mathrm{\psi}(\mathcal{C}_a^{n})]\Delta \mathrm{m}_{a} \leq &  {\lambda} \Delta t\sum_{a=1}^{\mathfrak{I}}\sum_{l=1}^{\mathfrak{I}}\sum_{j=1}^{\mathfrak{I}}(\mathrm{m}_{j}+\mathrm{m}_{l})\mathcal{C}_{j}^{n}\Delta \mathrm{m}_{j}\mathcal{C}_{l}^{n}\Delta \mathrm{m}_{l}\Delta \mathrm{m}_{a} \mathbb{B}(\mathrm{m}_{a},\mathrm{m}_{j},\mathrm{m}_{l})\mathrm{\psi}^{'}(\mathcal{C}_a^{n+1}).
\end{align*}
The convexity result in Lemma {\ref{lemma}} allows us to obtain
\begin{align}\label{Equi2}
\sum_{a=1}^{\mathfrak{I}} [\mathrm{\psi}(\mathcal{C}_a^{n+1})-\mathrm{\psi}(\mathcal{C}_a^{n})]\Delta \mathrm{m}_{a} & \leq  2 {\lambda} \Delta t\sum_{a=1}^{\mathfrak{I}}\sum_{l=1}^{\mathfrak{I}}\sum_{j=1}^{\mathfrak{I}}\mathrm{m}_{j}\mathcal{C}_{j}^{n}\Delta \mathrm{m}_{j}\mathcal{C}_{l}^{n}\Delta \mathrm{m}_{l}\Delta \mathrm{m}_{a}[ \mathrm{\psi}(\mathcal{C}_a^{n+1})+{\mathrm{\psi}}(\mathbb{B}(\mathrm{m}_{a},\mathrm{m}_{j},\mathrm{m}_{l}))] \nonumber \\
& \leq 2 {\lambda} \Delta t\sum_{a=1}^{\mathfrak{I}}\sum_{l=1}^{\mathfrak{I}}\sum_{j=1}^{\mathfrak{I}}\mathrm{m}_{j}\mathcal{C}_{j}^{n}\Delta \mathrm{m}_{j}\mathcal{C}_{l}^{n}\Delta \mathrm{m}_{l}\Delta \mathrm{m}_{a}\mathrm{\psi}(\mathcal{C}_a^{n+1})\nonumber\\
& +  2 {\lambda} \Delta t\sum_{a=1}^{\mathfrak{I}}\sum_{l=1}^{\mathfrak{I}}\sum_{j=1}^{\mathfrak{I}}\mathrm{m}_{j}\mathcal{C}_{j}^{n}\Delta \mathrm{m}_{j}\mathcal{C}_{l}^{n}\Delta \mathrm{m}_{l}\Delta \mathrm{m}_{a}{\mathrm{\psi}}(\mathbb{B}(\mathrm{m}_{a},\mathrm{m}_{j},\mathrm{m}_{l})).
\end{align}
After employing  Eq.(\ref{Tproperty}) and Eq.(\ref{36}) into the second term on RHS of the Eq.(\ref{Equi2}) lead to
\begin{align}\label{Equi3}
 2 {\lambda} \Delta t\sum_{a=1}^{\mathfrak{I}}\sum_{l=1}^{\mathfrak{I}}\sum_{j=1}^{\mathfrak{I}}\mathrm{m}_{j}\mathcal{C}_{j}^{n}\Delta& \mathrm{m}_{j}\mathcal{C}_{l}^{n}\Delta \mathrm{m}_{l}\Delta  \mathrm{m}_{a}{\mathrm{\psi}}(\mathbb{B}(\mathrm{m}_{a},\mathrm{m}_{j},\mathrm{m}_{l}))\nonumber \\ 
 &=  2 {\lambda} \Delta t\sum_{a=1}^{\mathfrak{I}}\sum_{l=1}^{\mathfrak{I}}\sum_{j=1}^{\mathfrak{I}} \mathrm{m}_{j}\mathcal{C}_{j}^{n}\Delta \mathrm{m}_{j}\mathcal{C}_{l}^{n}\Delta \mathrm{m}_{l}\Delta \mathrm{m}_{a}\frac{{\mathrm{\psi}}(\mathbb{B}(\mathrm{m}_{a},\mathrm{m}_{j},\mathrm{m}_{l}))}{\{\mathbb{B}(\mathrm{m}_{a},\mathrm{m}_{j},\mathrm{m}_{l})\}^{\theta}}{\mathbb{B}(\mathrm{m}_{a},\mathrm{m}_{j},\mathrm{m}_{l})}^{\theta}\nonumber\\ 
 &\leq 2{\lambda} \Delta t \mathfrak{R} \Pi_{\theta}(\psi) M_{1}^{in}{\|\mathbb{B}\|}^{\theta}_{\infty}\sum_{l=1}^{\mathfrak{I}}\mathcal{C}_{l}^{n} \Delta \mathrm{m}_{l}\nonumber \\
      &\leq 2\lambda  \Delta t \mathfrak{R} \Pi_{\theta}(\psi) M_{1}^{in}{\|\mathbb{B}\|}^{\theta}_{\infty} \|\mathcal{C}^{in}\|_{L^1}\,e^{2\lambda \mathfrak{R} \|\mathbb{B}\|_{L^{\infty}} M_{1}^{in} \mathfrak{T}}.
\end{align}
Now, Eq.(\ref{Equi2}) and Eq.(\ref{Equi3}) imply that
\begin{align*}
\sum_{a=1}^{\mathfrak{I}} [\mathrm{\psi}(\mathcal{C}_a^{n+1})-\mathrm{\psi}(\mathcal{C}_a^{n})]\Delta \mathrm{m}_{a}& \leq  2\lambda \Delta t M_{1}^{in}\|\mathcal{C}^{in}\|_{L^1}\,e^{2\lambda \mathfrak{R} \|\mathbb{B}\|_{L^{\infty}} M_{1}^{in} \mathfrak{T}}\sum_{a=1}^{\mathfrak{I}}\Delta \mathrm{m}_{a}\mathrm{\psi}(\mathcal{C}_a^{n+1})\nonumber \\
& + 2\lambda  \Delta t \mathfrak{R} \Pi_{\theta}(\psi)M_{1}^{in}{\|\mathbb{B}\|}^{\theta}_{\infty} \|\mathcal{C}^{in}\|_{L^1}\,e^{2\lambda \mathfrak{R} \|\mathbb{B}\|_{L^{\infty}} M_{1}^{in} \mathfrak{T}}.
\end{align*}
It can be easily simplified as
\begin{align*}
(1-2\lambda \Delta t M_{1}^{in}\|\mathcal{C}^{in}\|_{L^1}\,e^{2\lambda \mathfrak{R} \|\mathbb{B}\|_{L^{\infty}} M_{1}^{in} \mathfrak{T}})\sum_{a=1}^{\mathfrak{I}}\Delta \mathrm{m}_{a}\mathrm{\psi}(\mathcal{C}_a^{n+1}) \leq & \sum_{a=1}^{\mathfrak{I}}\Delta \mathrm{m}_{a}\mathrm{\psi}(\mathcal{C}_a^{n}) \nonumber \\
+ 2\lambda  \Delta t \mathfrak{R}& \Pi_{\theta}(\psi)M_{1}^{in}{\|\mathbb{B}\|}^{\theta}_{\infty} \|\mathcal{C}^{in}\|_{L^1}\,e^{2\lambda \mathfrak{R} \|\mathbb{B}\|_{L^{\infty}} M_{1}^{in} \mathfrak{T}}. 
\end{align*}
The inequality mentioned above suggests that 
\begin{align*}
\sum_{a=1}^{\mathfrak{I}}   \Delta \mathrm{m}_a \mathrm{\psi}(\mathcal{C}_a^{n+1}) \le A \sum_{a=1}^{\mathfrak{I}}   \Delta \mathrm{m}_a \mathrm{\psi}(\mathcal{C}_a^{n})+  B,
\end{align*} 
where 
$$  A= \frac{1}{(1- 2 \lambda \Delta t M_{1}^{in}\|\mathcal{C}^{in}\|_{L^1}\,e^{2\lambda \mathfrak{R} \|\mathbb{B}\|_{L^{\infty}} M_{1}^{in} \mathfrak{T}})},  \,\, 
B = \frac{2\lambda \Delta t \mathfrak{R} \Pi_{\theta}(\psi)M_{1}^{in}{\|\mathbb{B}\|}^{\theta}_{\infty} \|\mathcal{C}^{in}\|_{L^1}\,e^{2\lambda \mathfrak{R} \|\mathbb{B}\|_{L^{\infty}} M_{1}^{in} \mathfrak{T}}}{(1-2\lambda \Delta t M_{1}^{in}\|\mathcal{C}^{in}\|_{L^1}\,e^{2\lambda \mathfrak{R} \|\mathbb{B}\|_{L^{\infty}} M_{1}^{in} \mathfrak{T}})}.
$$
Therefore,
	\begin{align}\label{Equi4}
	\sum_{a=1}^{\mathfrak{I}}   \Delta \mathrm{m}_a \mathrm{\psi}(\mathcal{C}_a^{n}) \le A^{n} \sum_{a=1}^{\mathfrak{I}}   \Delta \mathrm{m}_a \mathrm{\psi}(\mathcal{C}_a^{in})+  B \frac{A^{n} -1}{A-1}.
	\end{align}
The result of (\ref{convex}) and Jensen's inequality allows us to 
		\begin{align}\label{Equi5}
	\int_0^{\mathbb{R}}  \mathrm{\psi}(\mathcal{C}^{\hslash}(t, \mathrm{m}))\,d\mathrm{m} \le & A^{n} \sum_{a=1}^{\mathfrak{I}}   \Delta \mathrm{m}_a \mathrm{\psi} \bigg( \frac{1}{\Delta \mathrm{m}_a } \int_{\wp_a}\mathcal{C}^{in}(\mathrm{m}) d\mathrm{m} \bigg)+  B \frac{A^{n} -1}{A-1}\nonumber\\
	\le & A^{n}\mathcal{I}+  B \frac{A^{n} -1}{A-1} < \infty, \quad \text{for all}\ \ \ t\in [0,\mathfrak{T}].
	\end{align}
	The computations for Case-(2), Case-(3) and Case-(4) are equivalent to the Case-(1). Only just, we got the different values of A and B, which are the following
	$$  A= \frac{1}{(1-  \lambda \Delta t M_{1}^{in}\|\mathcal{C}^{in}\|_{L^1}\,e^{2\lambda R \|\mathbb{B}\|_{L^{\infty}} M_{1}^{in} \mathfrak{T}})},  \,\, 
	B = \frac{\lambda \Delta t R \Pi_{\theta}(\psi)M_{1}^{in}{\|\mathbb{B}\|}^{\theta}_{\infty} \|\mathcal{C}^{in}\|_{L^1}\,e^{2\lambda R \|\mathbb{B}\|_{L^{\infty}} M_{1}^{in} \mathfrak{T}}}{(1-\lambda \Delta t M_{1}^{in}\|\mathcal{C}^{in}\|_{L^1}\,e^{2\lambda R \|\mathbb{B}\|_{L^{\infty}} M_{1}^{in} \mathfrak{T}})}.
	$$
	We can now assert that the family $(\mathcal{C}^{\hslash})$ is uniformly integrable thanks to the  De la Vallée Poussin theorem. Therefore, the sequence $(\mathcal{C}^{\hslash})$ is weakly relatively sequentially compact in $L^1(]0, \mathfrak{T}[\times ]0,\mathfrak{R}[)$ according to the Dunford-Pettis theorem, which combines this conclusion with equiboundedness. 
	\end{proof}	
	\begin{rem}
Given that a subsequence of $\mathcal{C}^{\hslash}$ converges weakly to $\mathcal{C}$ in $L^1$ as in Proposition \ref{equiintegrability}, and by using diagonal procedure, extraction of a subsequence is possible for $\mathbb{K}^{\hslash}$ and $\mathbb{B}^{\hslash}$ such that $\mathbb{K}^{\hslash} \rightarrow \mathbb{K}$ and $\mathbb{B}^{\hslash} \rightarrow \mathbb{B}$ as $\hslash\rightarrow 0.$
	\end{rem}
To prove the Theorem \ref{maintheorem} completely, we used the following point approximations on the computational domain as 
\begin{enumerate}
\item Right endpoint approximation: $
		\Xi^\hslash:\mathrm{m}\in (0,\mathfrak{R})\rightarrow
		\Xi^\hslash(\mathrm{m})=\sum_{a=1}^{\mathfrak{I}}\mathrm{m}_{a+1/2}\chi_{\wp_a}(\mathrm{m}).	$
	\item 	Left endpoint approximation:	$	\xi^\hslash:\mathrm{m}\in (0,\mathfrak{R})\rightarrow
		\xi^\hslash(\mathrm{m})=\sum_{a=1}^{\mathfrak{I}}\mathrm{m}_{a-1/2}\chi_{\wp_a}(\mathrm{m}).$
		\end{enumerate}
	
		We will employ the subsequent lemma, which follows from the Egorov and Dunford-Pettis theorems, to demonstrate the weak convergence of the approximated solutions.
\begin{lem}{\label{Wconverge}} [\cite{laurenccot2002continuous}]
Assuming a constant $\mathcal{L}$ and open subset $\varpi$ of $\mathbb{R}^m$, and  two sequences $(g_n)$ and
$(d_n)$ such that $(g_n)\in L^1(\varpi), g\in L^1(\varpi)$ and $$g_n\rightharpoonup g \in L^1(\varpi),\ \ \ \text{as}\ n\rightarrow \infty,$$ $(d_n)\in
L^\infty(\varpi), d \in L^\infty(\varpi),$ and for all $n\in
\mathbb{N}, |d_n|\leq \mathcal{L}$ with $$d_n\rightarrow d,\ \ \text{almost
everywhere in}\ \ \varpi  \ \text{as}\ \ n\rightarrow
\infty.$$ Then
$$\lim_{n\rightarrow \infty}\|g_n(d_n-d)\|_{L^1(\varpi)}=0$$
and $$g_n\, d_n\rightharpoonup g\, d\in L^1(\varpi),\ \ \ \text{as}\ n\rightarrow \infty.$$
\end{lem}\enter		
We have now accumulated  all  the evidences required to support Theorem \ref{maintheorem}. Let $\varphi\in C^1([0,\mathfrak{T}]\times ]0,\mathfrak{R}])$ be a test function with compact support with in $[0,t_{N-1}]$ for small $t$. Define the left endpoint approximation for the space variable of $\varphi$ on $\wp_a$ and the finite volume for the time variable over $\tau_n$ by
$$\varphi_a^n=\frac{1}{\Delta t}\int_{t_n}^{t_{n+1}}\varphi(t,\mathrm{m}_{a-1/2})dt.$$  

Multiplication of $\varphi_a^n$ with  (\ref{fully}) yields the following equation after employing the summation over $n$ \& $a$ 
	\begin{align}\label{convergence1}
	\sum_{n=0}^{N-1}\sum_{a=1}^{\mathfrak{I}}\Delta \mathrm{m}_{a} (\mathcal{C}_{a}^{n+1}-\mathcal{C}_{a}^{n})\varphi_a^n = &{\Delta t}\sum_{n=0}^{N-1}\sum_{a=1}^{\mathfrak{I}}\sum_{l=1}^{\mathfrak{I}}\sum_{j=a}^{\mathfrak{I}}\mathbb{K}_{j,l}\mathcal{C}_{j}^{n}\mathcal{C}_{l}^{n}\Delta \mathrm{m}_{j} \Delta \mathrm{m}_{l}\varphi_a^n \int_{\mathrm{m}_{a-1/2}}^{p_{j}^{a}}\mathbb{B}(\mathrm{m},\mathrm{m}_{j},\mathrm{m}_{l})d\mathrm{m}\nonumber \\
	&-{\Delta t}\sum_{n=0}^{N-1}\sum_{a=1}^{\mathfrak{I}}\sum_{j=1}^{\mathfrak{I}}\mathbb{K}_{a,j}\mathcal{C}_{a}^{n}\mathcal{C}_{j}^{n}\Delta \mathrm{m}_{a}\Delta \mathrm{m}_{j}\varphi_a^n. 		
					\end{align}
				The left-hand side (LHS) of the aforementioned equation, open for $n$ that leads to
	\begin{align*}
	\sum_{n=0}^{N-1}\sum_{a=1}^{\mathfrak{I}}\Delta \mathrm{m}_{a} (\mathcal{C}_{a}^{n+1}-\mathcal{C}_{a}^{n})\varphi_a^n = &	\sum_{n=0}^{N-1}\sum_{a=1}^{\mathfrak{I}}\Delta	\mathrm{m}_a  \mathcal{C}_a^{n+1}(\varphi_a^{n+1}-\varphi_a^n)+ \sum_{a=1}^{\mathfrak{I}}\Delta \mathrm{m}_a  \mathcal{C}_a^{in}\varphi_a^0\\
	=&\sum_{n=0}^{N-1}\sum_{a=1}^{\mathfrak{I}}\int_{\tau_{n+1}}\int_{\wp_a}\mathcal{C}^{\hslash}(t,\mathrm{m})\frac{\varphi(t,\xi^\hslash(\mathrm{m}))-\varphi(t-\Delta t,\xi^\hslash(\mathrm{m}))}{\Delta t}d\mathrm{m}\,dt \\&+\sum_{a=1}^{\mathfrak{I}}\int_{\wp_a}\mathcal{C}^{\hslash}(0,\mathrm{m})\frac{1}{\Delta t}\int_0^{\Delta t}\varphi(t,\xi^\hslash(\mathrm{m}))dt\,d\mathrm{m}\\
	= &\int_{\Delta t}^\mathfrak{T} \int_{0}^{\mathfrak{R}} \mathcal{C}^{\hslash}(t,\mathrm{m})\frac{\varphi(t,\xi^\hslash(\mathrm{m}))-\varphi(t-\Delta t,\xi^\hslash(\mathrm{m}))}{\Delta t}d\mathrm{m}\,dt\\
		& + \int_{0}^\mathfrak{R} \mathcal{C}^{\hslash}(0,\mathrm{m})\frac{1}{\Delta t}\int_0^{\Delta t}\varphi(t,\xi^\hslash(\mathrm{m}))dt\,d\mathrm{m}.
	\end{align*}
	  With the aid of Lemma \ref{Wconverge}, $\mathcal{C}^{\hslash}(0,\mathrm{m})\rightarrow \mathcal{C}^{in}$ in $L^1]0,\mathfrak{R}[$ will yield the following result since $\varphi\in C^1([0,\mathfrak{T}]\times ]0,\mathfrak{R}])$ has compact support and a bounded derivative				
	\begin{align}\label{finaltime1}
		\int_{0}^\mathfrak{R} \mathcal{C}^{\hslash}(0,\mathrm{m})\frac{1}{\Delta t}\int_0^{\Delta t}\varphi(t,\xi^\hslash(\mathrm{m}))dtd\mathrm{m}\rightarrow \int_{0}^\mathfrak{R} \mathcal{C}^{in}(\mathrm{m})\varphi(0,\mathrm{m})d\mathrm{m}
	\end{align}
as max$\{\hslash,\Delta t\}$ goes to $0$.
Now, applying Taylor series expansion of $\varphi$, Lemma \ref{Wconverge} and Proposition \ref{equiintegrability} confirm 
\begin{align*}
\int_0^\mathfrak{T}\int_0^\mathfrak{R} \mathcal{C}^{\hslash}(t,\mathrm{m})\frac{\varphi(t,\xi^\hslash(\mathrm{m}))-\varphi(t-\Delta
t,\xi^\hslash(\mathrm{m}))}{\Delta t}d\mathrm{m}\,dt \rightarrow \int_0^\mathfrak{T}\int_0^\mathfrak{R} \mathcal{C}(t,\mathrm{m})&\frac{\partial \varphi}{\partial t}(t,\mathrm{m})d\mathrm{m}\,dt,
\end{align*}
for max$\{\hslash,\Delta t\} \rightarrow 0$ and hence, we obtain
	\begin{align}\label{finaltime2}
	&\int_{\Delta t}^\mathfrak{T} \int_0^\mathfrak{R}
	\underbrace{\mathcal{C}^{\hslash}(t,\mathrm{m})\frac{\varphi(t,\xi^\hslash(\mathrm{m}))-\varphi(t-\Delta t,\xi^\hslash(\mathrm{m}))}{\Delta t}}_{\mathcal{C}(\varphi)} d\mathrm{m}\,dt \nonumber \\ 
	&= \int_0^\mathfrak{T} \int_0^\mathfrak{R} \mathcal{C}(\varphi)\, d\mathrm{m}\,dt -\int_0^{\Delta t} \int_0^\mathfrak{R} \mathcal{C}(\varphi)\, d\mathrm{m}\,dt  \rightarrow \int_0^\mathfrak{T} \int_0^\mathfrak{R}  \mathcal{C}(t,\mathrm{m})\frac{\partial \varphi}{\partial t}(t,\mathrm{m})d\mathrm{m}\,dt.
	\end{align}
Now, the first term in the RHS of Eq.(\ref{convergence1})	is taken for observing the computation
\begin{align}\label{convergence2}
{\Delta t}\sum_{n=0}^{N-1}\sum_{a=1}^{\mathfrak{I}}\sum_{l=1}^{\mathfrak{I}}\sum_{j=a}^{\mathfrak{I}}\mathbb{K}_{j,l}&\mathcal{C}_{j}^{n}\mathcal{C}_{l}^{n}\Delta \mathrm{m}_{j} \Delta \mathrm{m}_{l}\varphi_a^n \int_{\mathrm{m}_{a-1/2}}^{p_{j}^{a}}\mathbb{B}(\mathrm{m},\mathrm{m}_{j},\mathrm{m}_{l})d\mathrm{m} \nonumber \\
=&{\Delta t}\sum_{n=0}^{N-1}\sum_{a=1}^{\mathfrak{I}}\sum_{l=1}^{\mathfrak{I}}\mathbb{K}_{a,l}\mathcal{C}_{a}^{n}\mathcal{C}_{l}^{n}\Delta \mathrm{m}_{a} \Delta \mathrm{m}_{l}\varphi_a^n \int_{\mathrm{m}_{a-1/2}}^{\mathrm{m}_{a}}\mathbb{B}(\mathrm{m},\mathrm{m}_{a},\mathrm{m}_{l})d\mathrm{m} \nonumber \\
 &+ \Delta t\sum_{n=0}^{N-1}\sum_{a=1}^{\mathfrak{I}}\sum_{l=1}^{\mathfrak{I}}\sum_{j=a+1}^{\mathfrak{I}} \mathbb{K}_{j,l}\mathcal{C}_{j}^{n}\mathcal{C}_{l}^{n}\Delta \mathrm{m}_{j} \Delta \mathrm{m}_{l}\varphi_a^n \int_{\mathrm{m}_{a-1/2}}^{\mathrm{m}_{a+1/2}}\mathbb{B}(\mathrm{m},\mathrm{m}_{j},\mathrm{m}_{l})d\mathrm{m}.
\end{align}

The first term on the RHS of Eq.(\ref{convergence2}) simplifies to
\begin{align}\label{convergence3}
&\Delta t\sum_{n=0}^{N-1}\sum_{a=1}^{\mathfrak{I}}\sum_{l=1}^{\mathfrak{I}}\mathbb{K}_{a,l}\mathcal{C}_{a}^{n}\mathcal{C}_{l}^{n}\Delta \mathrm{m}_{a} \Delta \mathrm{m}_{l}\varphi_a^n \int_{\mathrm{m}_{a-1/2}}^{\mathrm{m}_{a}}\mathbb{B}(\mathrm{m},\mathrm{m}_{a},\mathrm{m}_{l})d\mathrm{m} \nonumber \\
&=\sum_{n=0}^{N-1}\sum_{a=1}^{\mathfrak{I}}\sum_{l=1}^{\mathfrak{I}}\int_{\tau_{n}}\int_{\Lambda_{a}^{h}}\int_{\wp_l}\mathbb{K}^{\hslash}(\mathrm{m},z)\mathcal{C}^{\hslash}(t,\mathrm{m})\mathcal{C}^{\hslash}(t,z)\varphi(t,\xi^\hslash(\mathrm{m}))\int_{\xi^{\hslash}(\mathrm{m})}^{X^{\hslash}(\mathrm{m})}\mathbb{B}(r,X^{\hslash}(\mathrm{m}),X^{\hslash}(z))dr\,dz\,d\mathrm{m}\,dt\nonumber \\
&= \int_{0}^{\mathfrak{T}}\int_{0}^{\mathfrak{R}}\int_{0}^{\mathfrak{R}}\mathbb{K}^{\hslash}(\mathrm{m},z)\mathcal{C}^{\hslash}(t,\mathrm{m})\mathcal{C}^{\hslash}(t,z)\varphi(t,\xi^\hslash(\mathrm{m}))\int_{\xi^{\hslash}(\mathrm{m})}^{X^{\hslash}(\mathrm{m})}\mathbb{B}(r,X^{\hslash}(\mathrm{m}),X^{\hslash}(z))dr\,dz\,d\mathrm{m}\,dt.
\end{align}
Next, the second term of Eq.(\ref{convergence2}) leads to
\begin{align}\label{convergence4}
\Delta t\sum_{n=0}^{N-1}\sum_{a=1}^{\mathfrak{I}}\sum_{l=1}^{\mathfrak{I}}\sum_{j=a+1}^{\mathfrak{I}} \mathbb{K}_{j,l}\mathcal{C}_{j}^{n}\mathcal{C}_{l}^{n}\Delta \mathrm{m}_{j} \Delta \mathrm{m}_{l}\varphi_a^n \int_{\mathrm{m}_{a-1/2}}^{\mathrm{m}_{a+1/2}}\mathbb{B}(\mathrm{m},\mathrm{m}_{j},\mathrm{m}_{l})d\mathrm{m} \nonumber \\
=\sum_{n=0}^{N-1}\sum_{a=1}^{\mathfrak{I}}\sum_{l=1}^{\mathfrak{I}}\sum_{j=a+1}^{\mathfrak{I}}\int_{\tau_{n}}\int_{\wp_a}\int_{\wp_l}\int_{\wp_j}\Big\{\mathbb{K}^{\hslash}(\mathrm{n},z)\mathcal{C}^{\hslash}(t,\mathrm{n})\mathcal{C}^{\hslash}(t,z)\varphi(t,\xi^\hslash(\mathrm{m}))& \nonumber \\ \frac{1}{\Delta \mathrm{m}_{i}}\int_{\wp_a}\mathbb{B}(r,X^{\hslash}(\mathrm{n}),X^{\hslash}(z))dr\Big\}\,d\mathrm{n}\,dz\,d\mathrm{m}\,dt \nonumber \\
= \int_{0}^{\mathfrak{T}}\int_{0}^{\mathfrak{R}}\int_{0}^{\mathfrak{R}}\int_{\Xi^{\hslash}(\mathrm{m})}^{\mathfrak{R}} \mathbb{K}^{\hslash}(\mathrm{n},z)\mathcal{C}^{\hslash}(t,\mathrm{n})\mathcal{C}^{\hslash}(t,z)\varphi(t,\xi^\hslash(\mathrm{m}))\mathbb{B}(X^{\hslash}(\mathrm{m}),X^{\hslash}(\mathrm{n}),X^{\hslash}(z))&d\mathrm{n}\,dz\,d\mathrm{m}\,dt.
\end{align}
Eqs.(\ref{convergence2})-(\ref{convergence4}), Lemma \ref{Wconverge} and Proposition \ref{equiintegrability} imply that as max$\{\hslash,\Delta t\} \rightarrow 0$ 
\begin{align}\label{convergence5}
{\Delta t}\sum_{n=0}^{N-1}\sum_{a=1}^{\mathfrak{I}}\sum_{l=1}^{\mathfrak{I}}\sum_{j=a}^{\mathfrak{I}}\mathbb{K}_{j,l}\mathcal{C}_{j}^{n}\mathcal{C}_{l}^{n}\Delta \mathrm{m}_{j} \Delta \mathrm{m}_{l}\varphi_a^n \int_{\mathrm{m}_{a-1/2}}^{p_{j}^{a}}\mathbb{B}(\mathrm{m},\mathrm{m}_{j},\mathrm{m}_{l})d\mathrm{m} \nonumber \\
\rightarrow \int_{0}^{\mathfrak{T}}\int_{0}^{\mathfrak{R}}\int_{0}^{\mathfrak{R}}\int_{\mathrm{m}}^{\mathfrak{R}} \mathbb{K}(\mathrm{n},z)\mathcal{C}(t,\mathrm{n})\mathcal{C}(t,z)\varphi(t,\mathrm{m})&\mathbb{B}(\mathrm{m},\mathrm{n},z)d\mathrm{n}\,dz\,d\mathrm{m}\,dt.
\end{align}
Now, the last term of Eq.(\ref{convergence1}) can be simplified as
	\begin{align}\label{convergence6}
	{\Delta t}\sum_{n=0}^{N-1}\sum_{a=1}^{\mathfrak{I}}\sum_{j=1}^{\mathfrak{I}}\mathbb{K}_{a,j}&\mathcal{C}_{a}^{n}\mathcal{C}_{j}^{n}\Delta \mathrm{m}_{a}\Delta \mathrm{m}_{j}\varphi_a^n	\nonumber\\
	&=\sum_{n=0}^{N-1}\sum_{a=1}^{\mathfrak{I}}\sum_{j=1}^{\mathfrak{I}}\int_{\tau_{n}}\int_{\wp_a}\int_{\wp_j}\mathbb{K}^{\hslash}(\mathrm{m},\mathrm{n})\mathcal{C}^{\hslash}(t,\mathrm{m})\mathcal{C}^{\hslash}(t,\mathrm{n})\varphi(t,\xi^\hslash(\mathrm{m}))d\mathrm{n}\,d\mathrm{m}\,dt \nonumber	\\
&\rightarrow \int_{0}^{\mathfrak{T}}\int_{0}^{\mathfrak{R}}\int_{0}^{\mathfrak{R}}\mathbb{K}(\mathrm{m},\mathrm{n})\mathcal{C}(t,\mathrm{m})\mathcal{C}(t,\mathrm{n})\varphi(t,\mathrm{m})d\mathrm{n}\,d\mathrm{m}\,dt,
	\end{align}
	for  max$\{\hslash,\Delta t\} \rightarrow 0$.
	Eqs.(\ref{convergence1})-(\ref{convergence6}) deliver the desired results for the weak convergence as presented in Eq.(\ref{convergence0}).

\section{Error Simulation }\label{Error}
In this section, the error estimation is explored for CBE, which is based on the idea of \cite{bourgade2008convergence}. Taking the uniform mesh is crucial for estimating the error component, i.e., $\Delta \mathrm{m}_a=\hslash$\,\,$\forall a\in \{1,2,\ldots, \mathfrak{I}\}$. The error estimate is achieved by providing an  estimation on the difference $\mathcal{C}^{\hslash}-\mathcal{C}$, where $\mathcal{C}^{\hslash}$ is constructed using the numerical technique and $\mathcal{C}$ represents the exact solution to the problem (\ref{maineq}).
By using the following theorem, we can determine the error estimate by making some assumptions about the kernels and the initial datum.
	\begin{thm}\label{errorth1}
	Assume that the collisional and breakage kernels hold $\mathbb{K}\in W^{1,\infty}_{loc}(\mathbb{R}^{+} \times \mathbb{R}^{+}),$ $\mathbb{B}\in W^{1,\infty}_{loc}(\mathbb{R}^{+} \times \mathbb{R}^{+}\times \mathbb{R}^{+})$ and  initial datum  $\mathcal{C}^{in}\in W^{1,\infty}_{loc}(\mathbb{R}^{+}).$ Moreover,  assuming a  uniform mesh and  $\Delta t$ that fulfill the condition (\ref{22}) imply that the error estimation
	\begin{align}\label{errorbound}
	\|\mathcal{C}^{\hslash}-\mathcal{C}\|_{L^{\infty}(0,\mathfrak{T};L^{1}]0, \mathfrak{R}[)} \leq \mathcal{H}(\mathfrak{T},\mathfrak{R})(\hslash+\Delta t)
	\end{align}
	where $\mathcal{H}(\mathfrak{T},\mathfrak{R})\geq 0$ is a constant  and $\mathcal{C}$ is the weak solution to (\ref{maineq}-\ref{initial}). 
	\end{thm}
Before proving the theorem, consider the following proposition, which provides  estimates on the approximate  and precise solutions $\mathcal{C}^{\hslash}$ and $\mathcal{C},$ respectively, given certain additional assumptions. These estimates are important in the analysis of  the error.
	\begin{propos}\label{bound2}
Considering the kernels $\mathbb{K} \in L^{\infty}_{loc}(\mathbb{R}^{+} \times \mathbb{R}^{+}), \mathbb{B} \in L^{\infty}_{loc}(\mathbb{R}^{+} \times \mathbb{R}^{+}\times \mathbb{R}^{+})$ and initial function $\mathcal{C}^{in}$ also exists in $L^{\infty}_{loc}$. Moreover, stability property (\ref{22}) follows for $\Delta t$, then the estimation on $\mathcal{C}^{\hslash}$ and $\mathcal{C}$ to (\ref{maineq}-\ref{initial}) as
$$ \|\mathcal{C}^{\hslash}\|_{L^{\infty}(]0,\mathfrak{T}[\times ]0, \mathfrak{R}[)}\leq \mathcal{H}(\mathfrak{T},\mathfrak{R}), \hspace{0.4cm} \|\mathcal{C}\|_{L^{\infty}(]0,\mathfrak{T}[\times ]0, \mathfrak{R}[)}\leq \mathcal{H}(\mathfrak{T},\mathfrak{R}). $$
Furthermore, if the kernels $\mathbb{K} \in    W^{1,\infty}_{loc}(\mathbb{R}^{+} \times \mathbb{R}^{+})$, $\mathbb{B}\in W^{1,\infty}_{loc}(\mathbb{R}^{+} \times \mathbb{R}^{+}\times \mathbb{R}^{+})$ and $\mathcal{C}^{in} \in W^{1,\infty}_{loc}(\mathbb{R}^{+})$, then 
\begin{align}\label{bound1}
\|\mathcal{C}\|_{W^{1,\infty}]0, \mathfrak{R}[} \leq \mathcal{H}(\mathfrak{T},\mathfrak{R}).
\end{align}
\end{propos}
\begin{proof}
The principal purpose of this proposition is to obtain a bound on  the solution $\mathcal{C}$ to Eq.(\ref{maineq}). In consequence, integrating Eq.(\ref{trunceq}) with respect to the time variable provides the following result
	\begin{align*}
\mathcal{C}(t,\mathrm{m})& \leq \mathcal{C}^{in}(\mathrm{m})+ \int_{0}^{t}\int_0^\mathfrak{R}\int_{\mathrm{m}}^{\mathfrak{R}} \mathbb{K}(\mathrm{n},z)\mathbb{B}(\mathrm{m},\mathrm{n},z)\mathcal{C}(s,\mathrm{n})\mathcal{C}(s,z)d\mathrm{n}\,dz\,ds \\ 
& \leq \mathcal{C}^{in}(\mathrm{m})+\|\mathbb{K}\|_{\infty}\|\mathbb{B}\|_{\infty}{\|\mathcal{C}\|}^{2}_{\infty,1}t,
\end{align*}	
	where $\|\mathcal{C}\|_{\infty,1}$ represents the norm of $\mathcal{C}$ in $L^{\infty}(0,\mathfrak{T}; L^{1}]0, \mathfrak{R}[\,)$. Hence,
\begin{align*}
\|\mathcal{C}\|_{L^{\infty}(]0,\mathfrak{T}[ \times ]0, \mathfrak{R}[)} \leq \mathcal{H}(\mathfrak{T},\mathfrak{R}).
\end{align*}
Now, let us go to the conclusion of an analysis of (\ref{bound1}). First, integrate Eq.(\ref{trunceq}) over the time variable $t$, and then differentiate it with respect to $\mathrm{m}$ which lead to
\begin{align*}
\frac{\partial \mathcal{C}(t,\mathrm{m})}{\partial \mathrm{m}} &=  \frac{\partial \mathcal{C}^{in}(\mathrm{m})}{\partial \mathrm{m}}+ \frac{\partial}{\partial \mathrm{m}}\biggl( \int_{0}^{t}\int_0^\mathfrak{R}\int_{\mathrm{m}}^{\mathfrak{R}} \mathbb{K}(\mathrm{n},z)\mathbb{B}(\mathrm{m},\mathrm{n},z)\mathcal{C}(s,\mathrm{n})\mathcal{C}(s,z)d\mathrm{n}\,dz\,ds\biggl) \\
& -\frac{\partial}{\partial \mathrm{m}}\biggl( \int_{0}^{t}\int_{0}^{\mathfrak{R}}\mathbb{K}(\mathrm{m},\mathrm{n})\mathcal{C}(t,\mathrm{m})\mathcal{C}(t,\mathrm{n})\,d\mathrm{n}\biggl).
\end{align*}
The following condition is produced by simplifying computations and using the maximum value over the $\mathrm{m}$ domain as follows
\begin{align*}
\left\|{\frac{\partial \mathcal{C}(t)}{\partial \mathrm{m}}}\right\|_{\infty} &\leq  \left\|{\frac{\partial \mathcal{C}^{in}(\mathrm{m})}{\partial \mathrm{m}}}\right\|_{\infty}+ [\|\mathbb{K}\mathbb{B}\|_{\infty}\|\mathcal{C}\|_{\infty,1}\|\mathcal{C}\|_{\infty}+ \|\mathbb{K}\|_{\infty}\|\mathbb{B}\|_{W^{1,\infty}}{\|\mathcal{C}\|}^{2}_{\infty,1}\\
& + \|\mathbb{K}\|_{W^{1,\infty}}\|\mathcal{C}\|_{\infty,1}\|\mathcal{C}\|_{\infty}
]t+ \|\mathbb{K}\|_{\infty}\|\mathcal{C}\|_{\infty,1}\int_{0}^{t}\left \|\frac{\partial \mathcal{C}}{\partial \mathrm{m}}\right\|_{\infty}\,ds,
\end{align*}
which can be written in a more coherent way
\begin{align*}
\left\|{\frac{\partial \mathcal{C}(t)}{\partial \mathrm{m}}}\right\|_{\infty}\leq \Upsilon (t)+\upsilon \int_{0}^{t}\left \|\frac{\partial \mathcal{C}}{\partial \mathrm{m}}\right\|_{\infty}\,ds,
\end{align*}
where 
$$\Upsilon (t)=\left\|{\frac{\partial \mathcal{C}^{in}(\mathrm{m})}{\partial \mathrm{m}}}\right\|_{\infty}+ [\|\mathbb{K}\mathbb{B}\|_{\infty}\|\mathcal{C}\|_{\infty,1}\|\mathcal{C}\|_{\infty}+ \|\mathbb{K}\|_{\infty}\|\mathbb{B}\|_{W^{1,\infty}}{\|\mathcal{C}\|}^{2}_{\infty,1}
 + \|\mathbb{K}\|_{W^{1,\infty}}\|\mathcal{C}\|_{\infty,1}\|\mathcal{C}\|_{\infty}
]t,$$
and
$$\upsilon =\|\mathbb{K}\|_{\infty}\|\mathcal{C}\|_{\infty,1}.$$
In addition, the following proof is established by the application of Gronwall's lemma and integration by parts:
\begin{align*}
\left\|{\frac{\partial \mathcal{C}(t)}{\partial \mathrm{m}}}\right\|_{\infty}  \leq  &\Upsilon (t)+ \int_{0}^{t}\Upsilon (s) \upsilon e^{\int_{s}^{t}\upsilon\,dr}\,ds\\	
 \leq &\Upsilon (0) e^{\upsilon t} + (\|\mathbb{K}\mathbb{B}\|_{\infty}\|\mathcal{C}\|_{\infty,1}\|\mathcal{C}\|_{\infty}+ \|\mathbb{K}\|_{\infty}\|\mathbb{B}\|_{W^{1,\infty}}{\|\mathcal{C}\|}^{2}_{\infty,1}\\
&+ \|\mathbb{K}\|_{W^{1,\infty}}\|\mathcal{C}\|_{\infty,1}\|\mathcal{C}\|_{\infty})[(e^{\upsilon t}-1)].
\end{align*}
Therefore,
\begin{align*}
\left\|{\frac{\partial \mathcal{C}}{\partial \mathrm{m}}}\right\|_{L^{\infty}(]0,\mathfrak{T}[ \times ]0, \mathfrak{R}[)} \leq \mathcal{H}(\mathfrak{T},\mathfrak{R}),
\end{align*}
which  concludes the result (\ref{bound1}). 
\end{proof}
Next, the RHS of  equation (\ref{fully}) is converted into the continuous form, and we obtain the first-order terms in that process which is explained in the following lemma. 
The discrete collisional birth-death part  in (\ref{fully}) is expressed as 
\begin{align}\label{errord}
B_{\mathcal{C}}(a)-D_{\mathcal{C}}(a)&=\frac{1}{\Delta \mathrm{m}_a}\sum_{l=1}^{\mathfrak{I}}\sum_{j=a}^{\mathfrak{I}}\mathbb{K}_{j,l}\mathcal{C}^{n}_{j}\mathcal{C}^{n}_{l}\Delta \mathrm{m}_{j}\Delta \mathrm{m}_{l}\int_{\mathrm{m}_{a-1/2}}^{p_{j}^{a}}\mathbb{B}(\mathrm{m},\mathrm{m}_{j},\mathrm{m}_{l})\,d\mathrm{m}-\sum_{j=1}^{\mathfrak{I}}\mathbb{K}_{a,j}\mathcal{C}^{n}_{a}\mathcal{C}^{n}_{j}\Delta \mathrm{m}_{j}.
\end{align}	
The subsequent lemma offers a simplified version of the preceding discrete terms.
\begin{lem}
Let $\mathcal{C}^{in}$ $\in W^{1,\infty}_{loc}$, and $\Delta \mathrm{m}_{a}=\hslash$ $\forall a$ be the uniform mesh. Also presuming that $\mathbb{K}$ and $\mathbb{B}$ satisfy the following criteria: $\mathbb{K},\mathbb{B}\in W^{1,\infty}_{loc}$, then for $(s,\mathrm{m})\in \tau_{n}\times \wp_i$, where $n\in \{0,1,\ldots, N-1\}$ and $a\in\{1,2,\ldots,\mathfrak{I}\}$ leads to
\begin{align}\label{convert1}
B_{\mathcal{C}}(a)-D_{\mathcal{C}}(a)&=\int_{0}^{\mathfrak{R}}\int_{{\Xi}^{h}(\mathrm{m})}^{\mathfrak{R}}\mathbb{K}^{\hslash}(\mathrm{n},z)\mathbb{B}^{\hslash}(\mathrm{m},\mathrm{n},z)\mathcal{C}^{\hslash}(s,\mathrm{n})\mathcal{C}^{\hslash}(s,z)\,d\mathrm{n}dz \nonumber \\ &-\int_{0}^{\mathfrak{R}}\mathbb{K}^{\hslash}(\mathrm{m},\mathrm{n})\mathcal{C}^{\hslash}(s,\mathrm{m})\mathcal{C}^{\hslash}(s,\mathrm{n})\,d\mathrm{n}+\varepsilon(\hslash).
\end{align}
In the strong $L^1$ topology, $\varepsilon(\hslash)$  defines the first order term with regard to $\hslash$: 
\begin{align}
\|\varepsilon(\hslash)\|_{L^1}\leq \frac{\|\mathbb{K}\mathbb{B}\|_{L^{\infty}}}{2}{\|\mathcal{C}^{in}\|}^{2}_{L^1}\,e^{4\lambda \mathfrak{R} \|\mathbb{B}\|_{L^{\infty}} M_{1}^{in} \mathfrak{T}} \hslash.
\end{align}
\end{lem}
\begin{proof}

With a uniform mesh, start with the discrete birth term of Eq.(\ref{errord}) and transform it into the continuous form for $\mathrm{m} \in \wp_a$,
\begin{align}\label{error1}
&\frac{1}{\Delta \mathrm{m}_a}\sum_{l=1}^{\mathfrak{I}}\sum_{j=a}^{\mathfrak{I}}\mathbb{K}_{j,l}\mathcal{C}^{n}_{j}\mathcal{C}^{n}_{l}\Delta \mathrm{m}_{j}\Delta \mathrm{m}_{l}\int_{\mathrm{m}_{a-1/2}}^{p_{j}^{a}}\mathbb{B}(\mathrm{m},\mathrm{m}_{j},\mathrm{m}_{l})\,d\mathrm{m} \nonumber \\
&= \sum_{l=1}^{\mathfrak{I}}\sum_{j=a+1}^{\mathfrak{I}}\mathbb{K}_{j,l}\mathcal{C}^{n}_{j}\mathcal{C}^{n}_{l}\Delta \mathrm{m}_{j}\Delta \mathrm{m}_{l}\frac{1}{\Delta \mathrm{m}_a}\int_{\mathrm{m}_{a-1/2}}^{\mathrm{m}_{a+1/2}}\mathbb{B}(\mathrm{m},\mathrm{m}_{j},\mathrm{m}_{l})\,d\mathrm{m}
+ \sum_{l=1}^{\mathfrak{I}}\mathbb{K}_{a,l}\mathcal{C}^{n}_{a}\mathcal{C}^{n}_{l}\Delta \mathrm{m}_{l}\int_{\mathrm{m}_{a-1/2}}^{\mathrm{m}_{a}}\mathbb{B}(\mathrm{m},\mathrm{m}_{a},\mathrm{m}_{l})\,d\mathrm{m} \nonumber \\
&= \int_{0}^{\mathfrak{R}}\int_{{\Xi}^{h}(\mathrm{m})}^{\mathfrak{R}}\mathbb{K}^{\hslash}(\mathrm{n},z)\mathbb{B}^{\hslash}(\mathrm{m},\mathrm{n},z)\mathcal{C}^{\hslash}(s,\mathrm{n})\mathcal{C}^{\hslash}(s,z)\,d\mathrm{n}dz+\varepsilon(\hslash),
\end{align}
where $\varepsilon(\hslash)= \sum_{l=1}^{\mathfrak{I}}\mathbb{K}_{a,l}\mathcal{C}^{n}_{a}\mathcal{C}^{n}_{l}\Delta \mathrm{m}_{l}\int_{\mathrm{m}_{a-1/2}}^{\mathrm{m}_{a}}\mathbb{B}(\mathrm{m},\mathrm{m}_{a},\mathrm{m}_{l})\,d\mathrm{m}$ is defined. Upon computing the $L^{1}$ norm of $\varepsilon(\hslash)$, the subsequent term is obtained
\begin{align*}
\|\varepsilon(\mathbb{F},\hslash)\|_{L^{1}}&\leq \|\mathbb{K}\mathbb{B}\|_{L^{\infty}}\sum_{a=1}^{\mathfrak{I}}\mathcal{C}_{a}^{n}\Delta \mathrm{m}_{a}\sum_{l=1}^{\mathfrak{I}}\mathcal{C}_{l}^{n}\Delta \mathrm{m}_{l}\int_{\mathrm{m}_{a-1/2}}^{\mathrm{m}_{a}}\,d\mathrm{m}\\\
& \leq \frac{\|\mathbb{K}\mathbb{B}\|_{L^{\infty}}}{2}{\big(\sum_{a=1}^{\mathfrak{I}}\mathcal{C}_{a}^{n}\Delta \mathrm{m}_{a}}\big)^{2}\hslash\\
& \leq \frac{\|\mathbb{K}\mathbb{B}\|_{L^{\infty}}}{2}{\|\mathcal{C}^{in}\|}^{2}_{L^1}\,e^{4\lambda \mathfrak{R} \|\mathbb{B}\|_{L^{\infty}} M_{1}^{in} \mathfrak{T}} \hslash.
\end{align*}	
Now, taking the discrete death term of (\ref{errord}) yields
\begin{align}\label{error2}
\sum_{j=1}^{\mathfrak{I}}\mathbb{K}_{a,j}\mathcal{C}^{n}_{a}\mathcal{C}^{n}_{j}\Delta \mathrm{m}_{j}=\int_{0}^{\mathfrak{R}}\mathbb{K}^{\hslash}(\mathrm{m},\mathrm{n})\mathcal{C}^{\hslash}(s,\mathrm{m})\mathcal{C}^{\hslash}(s,\mathrm{n})\,d\mathrm{n}.
\end{align}
Using the formula (\ref{convert1}), Eq.(\ref{trunceq}) and Eq.(\ref{fully}) exhibit a relation between $\mathcal{C}^{\hslash} $ and $\mathcal{C}$  for $t\in \tau_n$ as
\begin{align}\label{errorfull}
\int_{0}^{\mathfrak{R}}|\mathcal{C}^{\hslash}(t,\mathrm{m})-\mathcal{C}(t,\mathrm{m})|d\mathrm{m} &\leq  \int_{0}^{\mathfrak{R}}|\mathcal{C}^{\hslash}(0,\mathrm{m})-\mathcal{C}(0,\mathrm{m})|d\mathrm{m} + \sum_{\beta=1}^{3}(CB)_{\beta}(\hslash)\nonumber \\ 
&+\int_{0}^{\mathfrak{R}}|\epsilon(t,n)|\,d\mathrm{m}+ \|\varepsilon(\hslash)\|_{L^{1}} t,
\end{align}
where $(CB)_{\beta}(\hslash)$ represents the error terms for $\beta=1,2,3$ as
\begin{align*}
(CB)_{1}(\hslash)= \int_{0}^{t}\int_{0}^{\mathfrak{R}}\int_{0}^{\mathfrak{R}}\int_{\Xi^{h}(\mathrm{m})}^{\mathfrak{R}}|\mathbb{K}^{\hslash}(\mathrm{n},z)\mathbb{B}^{\hslash}(\mathrm{m},\mathrm{n},z)\mathcal{C}^{\hslash}(s,\mathrm{n})\mathcal{C}^{\hslash}(s,z)\\-\mathbb{K}(\mathrm{n},z)\mathbb{B}(\mathrm{m},\mathrm{n},z)\mathcal{C}(s,\mathrm{n})\mathcal{C}(s,z)|\,d\mathrm{n}\,dz\,d\mathrm{m}\,ds,
\end{align*}
\begin{align*}
(CB)_{2}(\hslash)= \int_{0}^{t}\int_{0}^{\mathfrak{R}}\int_{0}^{\mathfrak{R}}\int_{\mathrm{m}}^{\Xi^{h}(\mathrm{m})}\mathbb{K}(\mathrm{n},z)\mathbb{B}(\mathrm{m},\mathrm{n},z)\mathcal{C}(s,\mathrm{n})\mathcal{C}(s,z)\,d\mathrm{n}\,dz\,d\mathrm{m}\,ds,		
\end{align*}
and
\begin{align*}
(CB)_{3}(\hslash)= \int_{0}^{t}\int_{0}^{\mathfrak{R}}\int_{0}^{\mathfrak{R}}|\mathbb{K}^{\hslash}(\mathrm{m},\mathrm{n})\mathcal{C}^{\hslash}(s,\mathrm{m})\mathcal{C}^{\hslash}(s,\mathrm{n})-\mathbb{K}(\mathrm{m},\mathrm{n})\mathcal{C}(s,\mathrm{m})\mathcal{C}(s,\mathrm{n})|\,d\mathrm{n}\,d\mathrm{m}\,ds.
\end{align*}
Considering $|t-t_n|\leq \Delta t$,  the time discretization provides the following expression
\begin{align*}
\int_{0}^{\mathfrak{R}}|\epsilon(t,n)|\,d\mathrm{m} \leq & \int_{t_n}^{t}\int_{0}^{\mathfrak{R}}\int_{0}^{\mathfrak{R}}\int_{\Xi^{h}(\mathrm{m})}^{\mathfrak{R}}\mathbb{K}^{\hslash}(\mathrm{n},z)\mathbb{B}^{\hslash}(\mathrm{m},\mathrm{n},z)\mathcal{C}^{\hslash}(s,\mathrm{n})\mathcal{C}^{\hslash}(s,z)\,d\mathrm{n}\,dz\,d\mathrm{m}\,ds \nonumber \\
& + \int_{t_n}^{t}\int_{0}^{\mathfrak{R}}\int_{0}^{\mathfrak{R}}\mathbb{K}^{\hslash}(\mathrm{m},\mathrm{n})\mathcal{C}^{\hslash}(s,\mathrm{m})\mathcal{C}^{\hslash}(s,\mathrm{n})\,d\mathrm{n}\,d\mathrm{m}\,ds 
+ \int_{t_n}^{t}\int_{0}^{\mathfrak{R}} \varepsilon(\hslash)\,d\mathrm{m}\,ds.
\end{align*}
Given	$\mathbb{K}, \mathbb{B}\in W_{loc}^{1,\infty}$ and for $\mathrm{m},\mathrm{n}\in ]0,\mathfrak{R}],$ we have
	$$	|\mathbb{K}^{\hslash}(\mathrm{m},\mathrm{n})-\mathbb{K}(\mathrm{m},\mathrm{n})|\leq \|\mathbb{K}\|_{W^{1,\infty}} \hslash. $$
As a result, it produces an estimate of $(CB)_{1}(\hslash)$ using the $L^1$ bounds on $\mathcal{C}^{\hslash}$ and $\mathcal{C}$. To begin, divide the expression into four segments
\begin{align*}
(CB)_{1}(\hslash)&\leq  	\int_{0}^{t}\int_{0}^{\mathfrak{R}}\int_{0}^{\mathfrak{R}}\int_{0}^{\mathfrak{R}}|\mathbb{K}^{\hslash}(\mathrm{n},z)-\mathbb{K}(\mathrm{n},z)|\mathbb{B}(\mathrm{m},\mathrm{n},z)\mathcal{C}(s,\mathrm{n})\mathcal{C}(s,z)\,d\mathrm{n}\,dz\,d\mathrm{m}\,ds\\
& + \int_{0}^{t}\int_{0}^{\mathfrak{R}}\int_{0}^{\mathfrak{R}}\int_{0}^{\mathfrak{R}}\mathbb{K}^{\hslash}(\mathrm{n},z)|\mathbb{B}^{\hslash}(\mathrm{m},\mathrm{n},z)-\mathbb{B}(\mathrm{m},\mathrm{n},z)|\mathcal{C}(s,\mathrm{n})\mathcal{C}(s,z)\,d\mathrm{n}\,dz\,d\mathrm{m}\,ds\\
	& +  \int_{0}^{t}\int_{0}^{\mathfrak{R}}\int_{0}^{\mathfrak{R}}\int_{0}^{\mathfrak{R}} \mathbb{K}^{\hslash}(\mathrm{n},z)\mathbb{B}^{\hslash}(\mathrm{m},\mathrm{n},z)|\mathcal{C}^{\hslash}(s,\mathrm{n})-\mathcal{C}(s,\mathrm{n})|\mathcal{C}(s,z)\,d\mathrm{n}\,dz\,d\mathrm{m}\,ds  \\
	& + \int_{0}^{t}\int_{0}^{\mathfrak{R}}\int_{0}^{\mathfrak{R}}\int_{0}^{\mathfrak{R}}\mathbb{K}^{\hslash}(\mathrm{n},z)\mathbb{B}^{\hslash}(\mathrm{m},\mathrm{n},z)\mathcal{C}^{\hslash}(s,\mathrm{n})|\mathcal{C}^{\hslash}(s,z)-\mathcal{C}(s,z)|\,d\mathrm{n}\,dz\,d\mathrm{m}\,ds.
	\end{align*}
The above can be changed by applying Proposition \ref{bound2} and simplifying it to
	\begin{align}\label{error3}
	(CB)_{1}(\hslash) \leq &(\|\mathbb{K}\|_{W^{1,\infty}} \|\mathbb{B}\|_{\infty}+\|\mathbb{K}\|_{\infty} \|\mathbb{B}\|_{W^{1,\infty}})t  \mathfrak{R}^{3}  {\|\mathcal{C}\|}^{2}_{\infty}   \hslash \nonumber \\
	& + \mathfrak{R}^{2}\|\mathbb{K}\|_{\infty}\|\mathbb{B}\|_{\infty}(\|\mathcal{C}\|_{\infty}+ \|\mathcal{C}^{\hslash}\|_{\infty})\int_{0}^{t}\|\mathcal{C}^{\hslash}(s)-\mathcal{C}(s)\|_{L^1}\,ds,
	\end{align}
and 	similar estimation for $(CB)_{3}(\hslash)$ gives
	\begin{align}\label{error4}
	(CB)_{3}(\hslash) \leq \|\mathbb{K}\|_{W^{1,\infty}}t  \mathfrak{R}^{2}  {\|\mathcal{C}\|}^{2}_{\infty}   \hslash 
		 + \mathfrak{R}\|\mathbb{K}\|_{\infty}(\|\mathcal{C}\|_{\infty}+ \|\mathcal{C}^{\hslash}\|_{\infty})\int_{0}^{t}\|\mathcal{C}^{\hslash}(s)-\mathcal{C}(s)\|_{L^1}\,ds.
	\end{align}
	Finally, deal with the terms, $(CB)_{2}(\hslash)$ and $\int_{0}^{\mathfrak{R}}|\epsilon(t,n)|\,d\mathrm{m}$, it is clear that
	\begin{align}\label{error5}
	 (CB)_{2}(\hslash)\leq \frac{t \mathfrak{R}^2}{2}\|\mathbb{K}\mathbb{B}\|_{\infty}{\|\mathcal{C}\|}^{2}_{\infty} \hslash,
	\end{align}
	and 
	\begin{align}\label{error6}
	\int_{0}^{\mathfrak{R}}|\epsilon(t,n)|\,d\mathrm{m} \leq (\|\mathbb{K}\mathbb{B}\|_{\infty}{\|\mathcal{C}^{\hslash}\|}^{2}_{\infty}\mathfrak{R}^3+\|\mathbb{K}\|_{\infty}{\|\mathcal{C}^{\hslash}\|}^{2}_{\infty}\mathfrak{R}^2+\|\epsilon(\hslash)\|_{L^{1}})\Delta t.
	\end{align}
	Furthermore, substituting all the estimations (\ref{error3})-(\ref{error6}) in (\ref{errorfull}) and applying the Gronwall's lemma  conclude the result in (\ref{errorbound}).
\end{proof}
\section{Numerical Testing } \label{testing}
To test the problem's theoretical error estimation, in this part of the article, the discussion over experimental error and experimental order of convergence (EOC) are concluded for two combinations of collision kernel and breakage distribution function.

%
The experimental domain for the volume variable is taken as $[1e-3, 10]$ and computations are run from time 0 to 1.  In order to observe the EOC of the FVS in the absence of analytical results, the following double mesh relation is used:
	\begin{align}\label{approximateeoc}
	\text{EOC}=\ln \left( \frac{\|N_{\mathfrak{I}}-N_{2\mathfrak{I}}\|}{\|N_{2\mathfrak{I}}-N_{4\mathfrak{I}}\|} \right)/\ln(2).
	\end{align}
Here, $N_{\mathfrak{I}}$ stands for the total number of particles having $\mathfrak{I}$ cells, generated by scheme (\ref{fully}).
\\

\textbf{Test case 1:} Consider $\mathbb{K}(\mathrm{n},z)=\mathrm{n}z$ and $\mathbb{B}(\mathrm{m},\mathrm{n},z)=\delta(\mathrm{m}-0.4\mathrm{n})+\delta(\mathrm{m}-0.6\mathrm{n})$. The exponential decay initial condition ($e^{-\mathrm{m}}$) supports this problem. The computational domain is divided into uniform subintervals of 30, 60, 120, and 240 number of cells.  Table 1 represents the numerical errors and the EOC examined through Eq.(\ref{approximateeoc}). It is observed that the scheme provides the error in decreasing pattern as grid gets refined and it is of first-order convergence, as predicted by the theoretical results in Section \ref{Error}.


\begin{table}[!htb]
	           	
	           	  \centering
	           	    \begin{tabular}{ |p{1cm}|p{2cm}|p{1cm}|}
	           	   	            \hline
	           	   	            Cells      & Error & EOC\\
	           	   	                 \hline
	           	   	              30    & - & -  \\
	           	   	             	          \hline
	           	   	             	      60 & 0.4271$\times 10^{-4}$ & - \\
	           	   	             	           \hline
	           	   	             	       120  &   0.2006$\times 10^{-4}$   &1.0899   \\
	           	   	             	          \hline
	           	   	             	        240  & 0.0973$\times 10^{-4}$ & 1.0449  \\
	           	   	             	          \hline
	           	   	             	      480  & 0.0479$\times 10^{-4}$ & 1.0224   \\
	           	   	            \hline
	           	           \end{tabular}
	           	            \caption{ Test case 1}
	           	           \end{table}
\textbf{Test case 2:} For the second example, assume $ \mathbb{K}(\mathrm{n},z)=\mathrm{n}+z$ along with $\mathbb{B}(\mathrm{m},\mathrm{n},z)=\delta(\mathrm{m}-0.4\mathrm{n})+\delta(\mathrm{m}-0.6\mathrm{n})$. As analytical concentration function is unavailable in the literature, the error and the EOC are calculated via Eq.(\ref{approximateeoc}) for uniform cells with 30, 60, 120, 240, and 480 degrees of freedom. The computational volume domain and execution time are identical to the preceding test case. Again, according to Tables 2, the EOC of the FVS for this case is comparable to the first case, i.e., the first-order convergence is noticed.

\begin{table}[!htb]
	           	
	           	  \centering
	           	    \begin{tabular}{ |p{1cm}|p{2cm}|p{1cm}|}
	           	   	            \hline
	           	   	            Cells      & Error & EOC\\
	           	   	                 \hline
	           	   	             30    & - & -  \\
	           	   	            	           	           	          \hline
	           	   	            	           	           	      60 & 0.4309$\times 10^{-4}$ & - \\
	           	   	            	           	           	           \hline
	           	   	            	           	           	       120  &   0.2023$\times 10^{-4}$   &1.0905   \\
	           	   	            	           	           	          \hline
	           	   	            	           	           	        240  & 0.0981$\times 10^{-4}$ & 1.0452  \\
	           	   	            	           	           	          \hline
	           	   	            	           	           	      480  & 0.0483$\times 10^{-4}$ & 1.0226   \\
	           	   	            \hline
	           	           \end{tabular}
	           	            \caption{ Test case 2}
	           	           \end{table}

\section{Conclusions}\label{conclusion}
This article proposed a theoretical convergence analysis of  FVS  for solving the collisional breakage equation for the non-uniform mesh. It yielded a non-conservative scheme, for which a weak convergence analysis has been executed with unbounded collision and breakage distribution kernels. The result was accomplished in the presence of Weak $L^1$ compactness method based on Dunford-Pettis and La Vall$\acute{e}$e Poussin theorems.  In addition, explicit error estimation of the method was also explored for the kernels belonging to $ W^{1,\infty}_{loc}$. We further confirmed that the FVS is first-order accurate for uniform meshes by using two numerical cases of the model.
					
\section{Acknowledgments}
The work of Sanjiv Kumar Bariwal and Rajesh Kumar is supported by  CSIR India (file No. 1157/CSIR-UGC NET June 2019) and Science and Engineering Research Board (SERB), DST India (project MTR/2021/000866), respectively. 
		
\section*{Conflict of Interest}There are no conflicts of interest related to this work.

\bibliography{colref}

\begin{thebibliography}{10}

\bibitem{diemer2021applications}
R.~B. Diemer, ``Applications of the linear mass-sectional breakage population
  balance to various milling process configurations,'' {\em AAPS PharmSciTech},
  vol.~22, no.~3, pp.~1--17, 2021.

\bibitem{danha2015application}
G.~Danha, D.~Hildebrandt, D.~Glasser, and C.~Bhondayi, ``Application of basic
  process modeling in investigating the breakage behavior of ug2 ore in wet
  milling,'' {\em Powder Technology}, vol.~279, pp.~42--48, 2015.

\bibitem{spampinato2017modelling}
A.~Spampinato and D.~Axinte, ``On modelling the interaction between two
  rotating bodies with statistically distributed features: an application to
  dressing of grinding wheels,'' {\em Proceedings of the Royal Society A:
  Mathematical, Physical and Engineering Sciences}, vol.~473, no.~2208,
  p.~20170466, 2017.

\bibitem{chen2020collision}
S.~Chen and S.~Li, ``Collision-induced breakage of agglomerates in homogenous
  isotropic turbulence laden with adhesive particles,'' {\em Journal of Fluid
  Mechanics}, vol.~902, 2020.

\bibitem{kudzotsa2013mechanisms}
I.~Kudzotsa, {\em Mechanisms of aerosol indirect effects on glaciated clouds
  simulated numerically}.
\newblock University of Leeds, 2013.

\bibitem{yano2016explosive}
J.-I. Yano, V.~T. Phillips, and V.~Kanawade, ``Explosive ice multiplication by
  mechanical break-up in ice--ice collisions: a dynamical system-based study,''
  {\em Quarterly Journal of the Royal Meteorological Society}, vol.~142,
  no.~695, pp.~867--879, 2016.

\bibitem{cruger2016coefficient}
B.~Cr{\"u}ger, V.~Salikov, S.~Heinrich, S.~Antonyuk, V.~S. Sutkar, N.~G. Deen,
  and J.~Kuipers, ``Coefficient of restitution for particles impacting on wet
  surfaces: An improved experimental approach,'' {\em Particuology}, vol.~25,
  pp.~1--9, 2016.

\bibitem{lee2017development}
K.~F. Lee, M.~Dosta, A.~D. McGuire, S.~Mosbach, W.~Wagner, S.~Heinrich, and
  M.~Kraft, ``Development of a multi-compartment population balance model for
  high-shear wet granulation with discrete element method,'' {\em Computers \&
  Chemical Engineering}, vol.~99, pp.~171--184, 2017.

\bibitem{cheng1988scaling}
Z.~Cheng and S.~Redner, ``Scaling theory of fragmentation,'' {\em Physical
  Review Letters}, vol.~60, no.~24, p.~2450, 1988.

\bibitem{ziff1991new}
R.~M. Ziff, ``New solutions to the fragmentation equation,'' {\em Journal of
  Physics A: Mathematical and General}, vol.~24, no.~12, p.~2821, 1991.

\bibitem{ziff1985kinetics}
R.~M. Ziff and E.~McGrady, ``The kinetics of cluster fragmentation and
  depolymerisation,'' {\em Journal of Physics A: Mathematical and General},
  vol.~18, no.~15, p.~3027, 1985.

\bibitem{peterson1986similarity}
T.~W. Peterson, ``Similarity solutions for the population balance equation
  describing particle fragmentation,'' {\em Aerosol Science and Technology},
  vol.~5, no.~1, pp.~93--101, 1986.

\bibitem{breschi2017note}
G.~Breschi and M.~A. Fontelos, ``A note on the self-similar solutions to the
  spontaneous fragmentation equation,'' {\em Proceedings of the Royal Society
  A: Mathematical, Physical and Engineering Sciences}, vol.~473, no.~2201,
  p.~20160740, 2017.

\bibitem{hosseininia2006numerical}
E.~S. Hosseininia and A.~Mirghasemi, ``Numerical simulation of breakage of
  two-dimensional polygon-shaped particles using discrete element method,''
  {\em Powder Technology}, vol.~166, no.~2, pp.~100--112, 2006.

\bibitem{catak2010discrete}
M.~Catak, N.~Bas, K.~Cronin, J.~J. Fitzpatrick, and E.~P. Byrne, ``Discrete
  solution of the breakage equation using markov chains,'' {\em Industrial \&
  Engineering Chemistry Research}, vol.~49, no.~17, pp.~8248--8257, 2010.

\bibitem{liao2018discrete}
Y.~Liao, R.~Oertel, S.~Kriebitzsch, F.~Schlegel, and D.~Lucas, ``A discrete
  population balance equation for binary breakage,'' {\em International Journal
  for Numerical Methods in Fluids}, vol.~87, no.~4, pp.~202--215, 2018.

\bibitem{barik2018note}
P.~K. Barik and A.~K. Giri, ``A note on mass-conserving solutions to the
  coagulation-fragmentation equation by using non-conservative approximation,''
  {\em arXiv preprint arXiv:1802.08038}, 2018.

\bibitem{attarakih2009solution}
M.~M. Attarakih, C.~Drumm, and H.-J. Bart, ``Solution of the population balance
  equation using the sectional quadrature method of moments (sqmom),'' {\em
  Chemical Engineering Science}, vol.~64, no.~4, pp.~742--752, 2009.

\bibitem{ahmed2013stabilized}
N.~Ahmed, G.~Matthies, and L.~Tobiska, ``Stabilized finite element
  discretization applied to an operator-splitting method of population balance
  equations,'' {\em Applied Numerical Mathematics}, vol.~70, pp.~58--79, 2013.

\bibitem{lin2002solution}
Y.~Lin, K.~Lee, and T.~Matsoukas, ``Solution of the population balance equation
  using constant-number monte carlo,'' {\em Chemical Engineering Science},
  vol.~57, no.~12, pp.~2241--2252, 2002.

\bibitem{bourgade2008convergence}
J.-P. Bourgade and F.~Filbet, ``Convergence of a finite volume scheme for
  coagulation-fragmentation equations,'' {\em Math. Comp.}, vol.~77, no.~262,
  pp.~851--882, 2008.

\bibitem{bariwal2023convergence}
S.~K. Bariwal and R.~Kumar, ``Convergence and error estimation of weighted
  finite volume scheme for coagulation-fragmentation equation,'' {\em Numerical
  Methods for Partial Differential Equations}, vol.~39, no.~3, pp.~2561--2583,
  2023.

\bibitem{bariwal2023numerical}
S.~K. Bariwal, P.~K. Barik, A.~K. Giri, and R.~Kumar, ``Numerical study of
  finite volume scheme for coagulation--fragmentation equations with singular
  rates,'' {\em Journal of Hyperbolic Differential Equations}, vol.~20, no.~04,
  pp.~793--823, 2023.

\bibitem{kostoglou2000study}
M.~Kostoglou and A.~Karabelas, ``A study of the nonlinear breakage equation:
  analytical and asymptotic solutions,'' {\em Journal of Physics A:
  Mathematical and General}, vol.~33, no.~6, p.~1221, 2000.

\bibitem{laurenccot2001discrete}
P.~Lauren{\c{c}}ot and D.~Wrzosek, ``The discrete coagulation equations with
  collisional breakage,'' {\em Journal of Statistical Physics}, vol.~104,
  no.~1, pp.~193--220, 2001.

\bibitem{walker2002coalescence}
C.~Walker, ``Coalescence and breakage processes,'' {\em Mathematical Methods in
  the Applied Sciences}, vol.~25, no.~9, pp.~729--748, 2002.

\bibitem{barik2020global}
P.~K. Barik and A.~K. Giri, ``Global classical solutions to the continuous
  coagulation equation with collisional breakage,'' {\em ZAMP}, vol.~71, no.~1,
  pp.~1--23, 2020.

\bibitem{das2022existence}
A.~Das and J.~Saha, ``Existence and uniqueness of mass conserving solutions to
  the coagulation and collision-induced breakage equation,'' {\em The Journal
  of Analysis}, vol.~30, no.~3, pp.~1323--1340, 2022.

\bibitem{cheng1990kinetics}
Z.~Cheng and S.~Redner, ``Kinetics of fragmentation,'' {\em Journal of Physics
  A: Mathematical and General}, vol.~23, no.~7, p.~1233, 1990.

\bibitem{ernst2007nonlinear}
M.~H. Ernst and I.~Pagonabarraga, ``The nonlinear fragmentation equation,''
  {\em Journal of Physics A: Mathematical and Theoretical}, vol.~40, no.~17,
  p.~F331, 2007.

\bibitem{krapivsky2003shattering}
P.~Krapivsky and E.~Ben-Naim, ``Shattering transitions in collision-induced
  fragmentation,'' {\em Physical Review E}, vol.~68, no.~2, p.~021102, 2003.

\bibitem{giri2021weak}
A.~K. Giri and P.~Lauren{\c{c}}ot, ``Weak solutions to the collision-induced
  breakage equation with dominating coagulation,'' {\em J. Differential
  Equations}, vol.~280, pp.~690--729, 2021.

\bibitem{das2020approximate}
A.~Das, J.~Kumar, M.~Dosta, and S.~Heinrich, ``On the approximate solution and
  modeling of the kernel of nonlinear breakage population balance equation,''
  {\em SIAM journal on scientific computing}, vol.~42, no.~6, pp.~B1570--B1598,
  2020.

\bibitem{laurenccot2002continuous}
P.~Lauren{\c{c}}ot and S.~Mischler, ``The continuous coagulation-fragmentation
  equatons with diffusion,'' {\em Archive for Rational Mechanics and Analysis},
  vol.~162, no.~1, pp.~45--99, 2002.

\end{thebibliography}
\bibliographystyle{ieeetr}
				\end{document}